\documentclass[12pt,a4paper,reqno]{amsart}
\usepackage{a4wide,amssymb,epic,eepic,hyperref}

\numberwithin{equation}{section}
\newtheorem{thm}{Theorem}
\numberwithin{thm}{section}
\newtheorem{prop}[thm]{Proposition}
\newtheorem{lem}[thm]{Lemma}
\newtheorem{cor}[thm]{Corollary}

\theoremstyle{definition}
\newtheorem{defn}[thm]{Definition}
 
\theoremstyle{remark}
\newtheorem{rem}[thm]{Remark}

\newcounter{rmcount}\renewcommand{\thermcount}{{\rm\roman{rmcount}}}
\newenvironment{rmlist}{%
\begin{list}{{\rm(\thermcount)}}{\setlength{\labelwidth}{\leftmargin}%
\usecounter{rmcount}}}{\end{list}}

\newcounter{Rmcount}\renewcommand{\theRmcount}{{\rm\Roman{Rmcount}}}
\newenvironment{Rmlist}{%
\begin{list}{{\rm(\theRmcount)}}{\setlength{\labelwidth}{\leftmargin}%
\usecounter{Rmcount}}}{\end{list}}

\newcommand{\R}{\mathbb{R}} \newcommand{\N}{\mathbb{N}}              
\newcommand{\Z}{\mathbb{Z}} \newcommand{\C}{\mathbb{C}}      
\newcommand{\Rn}{{\R}^n}
\newcommand{\Rnl}{{\R}^{n-1}}

\providecommand{\cal}[1]{\mathcal{#1}}
\newcommand{\cs}{{\cal S}}              
\newcommand{\cd}{{\cal D}}              
\newcommand{\cf}{{\cal F}}
\newcommand{\cfi}{{\cal F}^{-1}}

\newcommand{\dual}[2]{\langle\,#1,\,#2\,\rangle}
\newcommand{\im}{\operatorname{i}}

\newcommand{\bspq}{B^{s,\vec a}_{\vec p,q}\,}           
\newcommand{\fspq}{F^{s,\vec a}_{\vec p,q}\,}           
\newcommand{\hsp}{{H}^{s,\vec{a}}_{\vec{p}}\,}
\newcommand{\wmp}{{W}^{\vec{m}}_{\vec{p}}\,}
\newcommand{\lpvec}{{L_{\vec{p}}\,}}

\newcommand{\supp}{\operatorname{supp}}
\newcommand{\op}[1]{\operatorname{#1}}

\newcommand{\imb}{\hookrightarrow}
\newcommand{\be}{\begin{equation}}
\newcommand{\ee}{\end{equation}}

\newcommand{\nrm}[2]{\|#1\|_{#2}}                       
\newcommand{\norm}[2]{\mathinner{\|}#1\,|#2\|}
\newcommand{\Norm}[2]{\mathinner{\bigl\|\,#1\,\big|#2\bigr\|}}

\title[On traces and mixed norms]
{On the trace problem for Lizorkin--Triebel spaces\\ with
mixed norms} 
\author[J.~Johnsen]{Jon Johnsen}
\address{\small Department of Mathematical Sciences,
 Aalborg University,
 Fredrik~Bajers Vej 7G,
 DK--9220 Aalborg East, Denmark }
\email{jjohnsen@math.aau.dk}
\author[W.~Sickel]{Winfried Sickel}
\address{\small Institute of Mathematics,
 Friedrich-Schiller-University Jena,
 Ernst-Abbe-Platz 1--2,
 D--07743 Jena, Germany} 
\email{sickel@minet.uni-jena.de}
\keywords{Traces of Sobolev spaces, Besov and Lizorkin--Triebel spaces, 
anisotropic spaces, mixed norms, maximal regularity%
\\[2\jot] {\tt Article appeared in 
Mathematische Nachrichten~{\bf 281\/}, no.~5 (2008), pp.~669--696.}
} 
\subjclass[2000]{46E35}

\begin{document}
 \begin{abstract}
  The subject is traces of
Sobolev spaces with mixed Lebesgue norms on Euclidean space. Specifically,
restrictions to the hyperplanes given by $x_1=0$ and 
$x_n=0$ are applied to functions
belonging to quasi-homogeneous, mixed-norm Lizorkin--Triebel
spaces $F^{s,\vec a}_{\vec p,q}$;
Sobolev spaces are obtained from these as special cases. 
Spaces admitting traces in the distribution sense
are characterised up to the borderline cases;
these are also covered in case $x_1=0$. For $x_1$ the trace spaces are proved
to be mixed-norm Lizorkin--Triebel spaces with a specific sum exponent; 
for $x_n$ they are
similarly defined Besov spaces.
The treatment includes continuous right-inverses and higher order traces.
The results rely on a sequence version of Nikol$'$skij's inequality,
Marschall's inequality for pseudodifferential operators (and Fourier
multiplier assertions), as well as dyadic ball criteria.
 \end{abstract}

\maketitle
%
\section{Introduction}
  \label{intr-sect}
\enlargethispage{\baselineskip}
\thispagestyle{empty}
The motivation for this paper comes from parabolic 
boundary problems. 
To settle ideas we consider a simple problem,
say for a domain $\Omega\subset\Rn$ with $C^\infty$
boundary $\Gamma:=\partial\Omega$, and
with $\Delta=\partial^2_1+\dots+\partial^2_n$ denoting the Laplacian,
\begin{align}
 \partial_t u -\Delta u &= f \quad\text{in $\Omega\times \,]0,T[\,$},
\\
 u_{|\Gamma}&=\varphi\quad\text{on $\Gamma\times\,]0,T[\,$},
  \label{uGamma-eq}
\\
 u_{|t=0}&=u_0 \quad\text{at $\Omega\times\{0\}$}.
\end{align}
Among the data, 
$f(x,t)$ may have different integrability properties with respect to the 
$x$- and $t$-directions. E.g.\ there may be given $p_1\ne p_2$ in 
$[1,\infty]$ such that
\begin{equation}
  \big(\int_0^T (
  \int_{\Omega} |f(x,t)|^{p_1}\,dx
  )^{p_2/p_1} \,dt \big)^{1/p_2}
  <\infty.
\end{equation}
(It is throughout understood that an $L_\infty$-norm applies whenever $p_j=\infty$.)

Correspondingly, any solution $u(x,t)$ is expected to belong to this 
$L_{\vec p}$ space, $\vec p=(p_1,p_2)$, at least if $\varphi=0$ and $u_0=0$. 
It is well known that this can have various interpretations such as
a bounded kinetic energy of the associated physical system for 
$\vec p=(2,\infty)$.
When $Q_T=\Omega\times\,]0;T[\,$, a more precise information on $u$ will be
that 
\begin{equation}
  u,\ \partial_t u,\ \partial^2_{x_1}u,\dots,\partial^2_{x_n}u
  \in L_{\vec p}(Q_T).
\end{equation}
The set of such $u$ is denoted $W^{2,1}_{\vec p}(Q_T)$.
That in this case 
$u\in W^{2,1}_{\vec p}(Q_T)$ is a result of the maximal regularity
theory, that has been intensively studied since the 1980s; 
the reader may consult \cite[Ch.~III,4.10]{Am95} as a reference to this
development.

In case $\varphi\ne0$ and $u_0\ne0$,
a natural question is of course 
in which spaces it is possible to prescribe $\varphi$ and $u_0$, such that 
$u\in  W^{2,1}_{\vec p}(Q_T)$ still holds.
Even for the above heat problem, 
the answer is somewhat delicate for $p_1\ne p_2$.

This investigation was seemingly begun by Weidemaier \cite{Wei98,Wei02,Wei05}, 
but other works have been devoted to this area, 
cf.\ the paper by Denk, Hieber and Pr{\"u}ss \cite{DeHiPr}.

To give a brief account of what can be expected, 
let $\gamma_0$ denote the operator of restriction to the lateral surface, 
so that the boundary condition \eqref{uGamma-eq} 
may be written $\gamma_0 u=\varphi$,
and let $r_0$ stand for the restriction to the initial surface at
$t=0$ (i.e.\ $r_0 u=u_0$). 

However, we simplify by taking the flat case in which 
$\Omega=\Rn$ and $t\in\R$.
The initial data $u_0$ should then be given in the Besov space
$B^{2-2/p_2}_{p_1,p_2}(\Rn)$, as $r_0$ is a surjection
\begin{equation}
  r_0\colon W^{2,1}_{\vec p} (\Rn\times\R)\to B^{2-2/p_2}_{p_1,p_2}(\Rn).
\end{equation}

For $\varphi$ the situation is different, for if 
$\Rnl_x\times\R_t$ is equipped with mixed-norm spaces
$L_{p'}(\Rnl_x\times\R_t)$ for $p'=(p_1,\dots,p_1,p_2)$ 
($n-1$ copies of $p_1$),
$\gamma_0$ is a surjection
\begin{equation}
  \gamma_0\colon W^{2,1}_{\vec p} (\Rn\times\R)
      \to F^{2-1/p_1,a'}_{p',p_1}(\Rnl_x\times\R_t).
  \label{WF-eq}
\end{equation}
Here the range space is a Lizorkin--Triebel space with mixed norms
(due to $p'$) and with its sum exponent equal to $p_1$ 
(so in general this is not a Besov space).
In addition the space has an anisotropy related to the 
smoothness index $s$; this is handled via weights $a_j$
assigned to each coordinate axis, so that $a'=(1,\dots,1,2)$. 
The resulting quasi-homogeneity of the space
is well known, so the exact definitions are given
in Section~\ref{mixd-sect} below.

\bigskip

Motivated by the above outline, we shall study the general trace problem for 
the quasi-homogeneous, mixed-norm Lizorkin--Triebel spaces $\fspq(\Rn)$.
This problem was first studied by Berkolaiko \cite{Ber84,Ber85,Ber87I,Ber87II}.
The fact that $\gamma_0$ has a Lizorkin--Triebel space
as the range was discovered by him for spaces with
$1<p_k<\infty$ for all $k$, $1<q<\infty$. 

Like Berkolaiko, our 
point of departure is a Littlewood--Paley decomposition of the functions,
$u=\sum u_j$,
but this we combine with a rather straightforward 
$L_\infty$--$L_{\vec p}\,$-estimate, using 
maximal functions $u^*_j$ of Peetre--Fefferman--Stein type. More precisely,
if $\vec p=(p_1,p'')$,
\begin{equation}
  \sup_{z\in \R}\Norm{(\sum_{j=0}^{\infty} 
                2^{j(s-\frac{a_1}{p_1})p_1}
                |u_j(z,\cdot)|^{p_1})^{\frac1{p_1}}}{L_{p''}(\Rnl)}
  \le
  c\Norm{ \sup_{j=0,1,2,\dots} 2^{sj} |u^*_j(\cdot)|}{L_{\vec p}(\Rn)}.
\end{equation}
The expression to the right is estimated by $\nrm{u}{}$ in $\fspq$,
so most of the conclusions can be drawn from this $L_\infty$--$L_p$-estimate.
With this method, there are extensions to arbitrary
$p_k\in\,]0,\infty[\,$, for all $k$, $0<q\le\infty$. In particular we 
settle the cases when $p_k=1$ for one or more $k=1,\dots,n$, 
which the previous works on the subject 
\cite{Ber84,Ber85,Ber87I,Ber87II,DeHiPr,Wei05}
were unable to cover.

Moreover, the trace of $\fspq(\Rn)$ is treated for all $s$ above a certain
limit. The isotropic condition $s>\frac1p$ is for mixed norms replaced by 
$s>\frac1{p_k}$ for the trace at $x_k=0$, when all $p_j\in \,]1,\infty[\,$.
As a minor novelty a shift of the borderline is necessary if
$0<p_j<1$ holds for one the tangential variables $x_j$. This is evident from
\eqref{spq-cnd} in Theorem~\ref{main1-thm} and Figure~\ref{borderline_1-fig} 
below.

\bigskip

The paper is organised as follows:
In 
Section~\ref{main-sect}   
our results on the trace problems are presented.
The definition of $\fspq$ is recalled in
Section~\ref{mixd-sect}, together with the properties needed 
for the spaces. In the definition we follow Triebel's book \cite{T2}, 
though the 
conventions for the quasi-homogeneity given by $\vec a$ are the same as 
in \cite{Y1} (and as in our joint work with Farkas on 
the unmixed cases \cite{FaJoSi00}); 
mixed norms are treated as in works of
Schmeisser, Schmeisser and Triebel \cite{Sch84,ScTr87}, 
but here we also draw on a joint work \cite{JoSi}
proving a crucial Nikol$'$skij inequality for vector-valued functions. 
In addition
dyadic corona and ball criteria for the $\fspq$ are established
in the applicable style known at least since \cite{Y1}; 
a pointwise estimate of pseudo-differential operators is also shown,
inspired by a work of Marschall \cite{Mar96}.
Section~\ref{prfs-sect} 
then proceeds to give the proofs, using maximal functions 
(based on an estimate of Bagby \cite{Bag75});
Section~\ref{finl-sect} contains a few final remarks.

\section{Traces of quasi-homogeneous mixed-norm Lizorkin--Triebel spaces}
  \label{main-sect}

\subsection{The main theorems}
In the following vectors $\vec x=(x_1,\dots,x_n)$ in $\Rn$ may be split in 
groups like $\vec x=(x',x_k,x'')$. E.g.\ when restriction to
the hyperplane
$\Gamma_k$ given by $x_k=0$ is considered, $x'=(x_1,\dots,x_{k-1})$ and 
$x''=(x_{k+1},\dots,x_n)$ will be convenient;
because $x'$ and $x''$ both indicate tuples, vector arrows are suppressed.
These conventions are also used for $\vec a$ and $\vec p$.

In general one can define many standard traces, say for $f\in C^\infty(\Rn)$,
\begin{equation}
  \gamma_{j,k} f(x',x'')=\frac{\partial^j f}{\partial
x_k^j}(x',x_k,x'')\bigm|_{x_k=0}.
\end{equation}
Here we shall mainly treat $\gamma_{0,k}$ 
for $k=1$ and $k=n$. 
However, for general $f$, the operator $\gamma_{0,k}$ should be understood
as the distributional trace defined in the natural way as
$\gamma_{0,k}f=f|_{x_k=0}$ when $f$ in its dependence of $x_k$
defines a continuous map from $\R$ to $\cd'(\Rnl)$; that is, $\gamma_{0,k}$ 
is defined for $f$ in the subspace
\begin{equation}
    C(\R_{x_k},\cd'(\Rnl))\subset\cd'(\Rn).
\end{equation}
Here we recall that any $g\in C(\R_{x_k},\cd'(\Rnl))$ defines a distribution
$\Lambda_g$ in $n$ variables, with its action on arbitrary
$\varphi\in C^\infty_0(\Rn)$ given by integration of the continuous function
$x_k\mapsto \dual{g(x_k)}{\varphi(\cdot,x_k,\cdot)}$; more precisely,
$\dual{\Lambda_g}{\varphi}=
\int_\R \dual{g(x_k)}{\varphi(\cdot,x_k,\cdot)}\,dx_k$.
For topological vector spaces $X$, $Y$, the set of continuous
bounded maps $f\colon X\to Y$ is denoted by $C_{\op{b}}(X,Y)$.

All mapping properties of $\gamma_{0,k}$ are meant as restrictions,
for example
$\gamma_{0,k}\colon X\to Y$ means that for the distributional trace,
$X$ is contained in the preimage $\gamma_{0,k}^{-1}(Y)$.

Similarly
$\gamma_{j,k}f$ is defined for $f\in \cd'(\Rn)$ when the
distributional derivative
$\partial^j_{x_k}f$ is in $C(\R_{x_k},\cd'(\Rnl))$.

\bigskip

As our first main result, we determine the 
$\fspq$ that belong to the domain of the trace in the inner variable:

\begin{thm}   \label{main1-thm}
For a given 
anisotropy $\vec a=(a_1,\dots,a_n)\in \,]0,\infty[\,^n$,
let $\vec p\in \,]0,\infty[\,^n$
while 
$0<q\le\infty$ and $s\in \R$.
For the trace $\gamma_{0,1}$ on the hyperplane 
$\{x_1=0 \}$ the following properties of a triple
$(s,\vec p, q)$ are equivalent: 
\begin{rmlist}
  \item
  \label{spq-cnd}
 $(s,\vec p,q)$ satisfies the inequality 
\begin{equation}
  s\ge\frac{a_1}{p_1}+\sum_{k>1}(\frac{a_k}{p_k}-a_k)_+,
  \label{spq1-eq}
\end{equation}
and, in addition, $s=\frac{a_1}{p_1}$ only holds if also $p_1\le1$;
  \item
  \label{cd'-cnd}
    the operator $\gamma_{0,1}$ is continuous from $\fspq(\Rn)$ to 
    $\cd'(\Rnl)$.
\end{rmlist}
In the affirmative case there is a continuous embedding 
$\fspq(\Rn)\imb C_{\op{b}}(\R_{x_1},L_{r''}(\Rnl))$, with the integral
exponents given by $r_k=\max(1,p_k)$ for $k=2,\dots,n$.
\end{thm}

The co-domain $\cd'$ above is of course not optimal. Indeed, it is a main point
for $\gamma_{0,1}$ that
the range space is a \emph{Lizorkin--Triebel} space; cf.~\eqref{WF-eq}.
This result is established here under the condition that
\begin{equation}
  s>\frac{a_1}{p_1}+\sum_{k\ge2} (\frac{a_k}{\min(1,p_2,\dots,p_k,q)}-a_k).
  \label{spq-eq}
\end{equation}
This is stronger than the sharp inequality in
\eqref{spq-cnd}, but e.g.\ when $q$, $p_k\ge1$ for all $k>1$ 
it gives the same borderline as \eqref{spq-cnd}; in general it does so 
if $q\ge p_1\ge\dots\ge p_n$.

\begin{thm}   \label{main1-thm'}
Let $\vec p\in \,]0,\infty[\,^n$, $0<q\le\infty$ and $\vec a\in\,]0,\infty[\,^n$.
When $(s,\vec p,q)$ fulfils \eqref{spq-eq},
then $\gamma_{0,1}$ is a bounded \emph{surjection} 
$\fspq(\Rn)\to F^{s-\frac{a_1}{p_1},a''}_{p'',p_1}(\Rnl)$.
\end{thm}

The implication \eqref{cd'-cnd}$\implies$\eqref{spq-cnd} in
Theorem~\ref{main1-thm} is actually a consequence of
the following result, that we obtain from specific counterexamples.

\begin{lem}
  \label{spq-lem}
Let $m\in \{1,\dots,n\}$. If $\gamma_{0,m}$ is continuous
$\fspq(\Rn)\to \cd'(\Rnl)$, then it holds that 
$s\ge\tfrac{a_m}{p_m}+\sum_{k\ne m}
(\tfrac{a_k}{p_k}-a_k)_+$. In case
$s=\tfrac{a_m}{p_m}$ (so that $p_k\ge1$ for all $k\ne m$) continuity 
of $\gamma_{0,m}$ implies $p_m\le1$.
\end{lem}

In connection with restriction to the hyperplane given by $x_n=0$, 
our result corresponding to Theorem~\ref{main1-thm} leaves a borderline case
open in the quasi-Banach space case.

\begin{thm}   \label{mainn-thm}
For a given anisotropy
$\vec a\in\,]0,\infty[\,^n$,
let $\vec p\in \,]0,\infty[\,^n$, $0<q\le\infty$ and $s\in \R$.
For the trace $\gamma_{0,n}$ on $\{x_n=0 \}$
it holds for the following properties of a triple
$(s,\vec p, q)$ that \eqref{spq-cnd'}$\implies$\eqref{cd'-cnd'}:
\begin{rmlist}
  \item
  \label{spq-cnd'}
 $(s,\vec p,q)$ satisfies 
\begin{equation}
  s\ge\frac{a_n}{p_n}+\sum_{k<n}(\frac{a_k}{p_k}-a_k)_+
  \label{spq-ineq'}
\end{equation}
and, in addition, equality only holds if $p_n\le1$;
  \item
  \label{cd'-cnd'}
    the operator $\gamma_{0,n}$ is continuous from $\fspq(\Rn)$ to 
    $\cd'(\Rnl)$.
\end{rmlist}
Conversely \eqref{cd'-cnd'}$\implies$\eqref{spq-cnd'}
in case $p_k\ge1$ for all $k<n$; and if $0<p_k<1$ for some
$k\in \{1,\dots,n-1\}$, then \eqref{cd'-cnd'} 
implies the inequality \eqref{spq-ineq'}.

When \eqref{spq-cnd'} holds, 
then $\fspq(\Rn)\imb C_{\op{b}}(\R,L_{r'}(\Rnl))$ 
with $r_k=\max(1,p_k)$
for $k\le n-1$.
\end{thm}
Here the implications of \eqref{cd'-cnd'} are obtained from 
Lemma~\ref{spq-lem} for $m=n$.

\bigskip

For the trace $\gamma_{0,n}$, that acts in the outer integration variable, 
the range is generically a \emph{Besov} space:

\begin{thm}   \label{mainn-thm'}
Let $\vec p\in \,]0,\infty[\,^n$, $0<q\le\infty$ and $\vec
a\in\,]0,\infty[\,^n$. 
When the triple $(s,\vec p,q)$  fulfils
\begin{equation}
  s>\frac{a_n}{p_n}+\sum_{k<n} (\frac{a_k}{\min(1,p_1,\dots,p_k)}-a_k),
  \label{spqn-cnd}
\end{equation}
then $\gamma_{0,n}$ is a bounded \emph{surjection} 
$\fspq(\Rn)\to B^{s-\frac{a_n}{p_n},a'}_{p',p_n}(\Rnl)$.
\end{thm}

Since $F^s_{p,p}=B^s_{p,p}$ in the isotropic case, we get 
for $s>\tfrac{1}{p}$, $1<p<\infty$ that
\begin{equation}
  \gamma_{0,1}(F^s_{p,q})= F^{s-1/p}_{p,p}=B^{s-1/p}_{p,p}
  =\gamma_{0,n}(F^s_{p,q}).
\end{equation}
In this way the present results give back the isotropic trace theory,
and they show how things split up qualitatively 
(with $F$- and $B$-spaces as ranges) 
and quantitatively (with $p_1$ and $p_n$ as sum exponents)
when mixed norms are introduced.

\bigskip

In Theorems~\ref{main1-thm'} and \ref{mainn-thm'} the surjectivity was 
just a convenient way to express the optimality of taking
$F^{s-\frac{a_1}{p_1},a''}_{p'',p_1}$ and 
$B^{s-\frac{a_n}{p_n},a'}_{p',p_n}$, respectively, as co-domains. But not
surprisingly the stronger fact that $\gamma_{0,1}$
and $\gamma_{0,n}$ have everywhere defined right-inverses also holds in the
present context. 

\begin{thm}
  \label{K-thm}
There exist continuous operators $K_1$, $K_n\colon \cs'(\Rnl)\to
\cs'(\Rn)$, both with range in the space $C_{\op{b}}(\R,\cs'(\Rnl))$, 
such that for every $v\in \cs'(\Rnl)$,
\begin{equation}
  \gamma_{0,1}(K_1 v)= v, \qquad
  \gamma_{0,n}(K_n v)= v.
  \label{gK-eq}
\end{equation}
Moreover,
for any $\vec p=(p_1,\dots,p_n)$ in $\,]0,\infty[\,^n$ and any 
$\vec a$, 
\begin{alignat}{2}
  K_1&\colon F^{s,a''}_{p'',p_1}(\Rnl)&\to F^{s+\frac{a_1}{p_1},\vec a}_{\vec
p,q}(\Rn) &\quad\text{for}\quad 0<q\le\infty,
  \label{K1F-eq}  \\
  K_n&\colon B^{s,a'}_{p',p_n}(\Rnl)&\to F^{s+\frac{a_n}{p_n},\vec a}_{\vec
p,q}(\Rn) &\quad\text{for}\quad 0<q\le\infty,
  \label{KnF-eq}  
\end{alignat}
are bounded maps for arbitrary $s\in\R$.
\end{thm}

\bigskip

Let us also briefly describe results for higher order traces $\gamma_{j,k}$.
Because they are composites of the trace $\gamma_{0,k}$ and differentiation 
$\partial^{j}_{x_k}$, both in the sense of distributions, and since
$\partial^{j}_{x_k}$ has order $ja_k$ in the $\fspq$-scale, the
continuity properties of $\gamma_{j,k}$ are straightforward consequences of
the above theorems. 

As usual, the surjectivity of $\gamma_{j,k}$ is implied by that of the
matrix-formed operator $\rho_{m,k}$ used for posing Cauchy problems,
\begin{equation}
  \rho_{m,k}=
  \left(\begin{smallmatrix}
  \gamma_{0,k}\\ \gamma_{1,k}\\ \vdots \\ \gamma_{m-1,k}
  \end{smallmatrix}\right).
  \label{rhom-eq}
\end{equation}
Under the assumptions 
$\vec p\in \,]0,\infty[\,^n$, $0<q\le\infty$ and $\vec a\in\,]0,\infty[\,^n$
as before, the following holds:
\begin{cor}   \label{gj1-cor}
When
$s>(m-1)a_1+\frac{a_1}{p_1}+\sum_{k>1} 
(\frac{a_k}{\min(1,p_2,\dots,p_k,q)}-a_k)$
then $\rho_{m,1}$ is a bounded surjection
\begin{equation}
  \rho_{m,1}\colon \fspq(\Rn)\to \prod_{j=0}^{m-1}
  F^{s-ja_1-\frac{a_1}{p_1},a''}_{p'',p_1}(\Rnl).
  \label{rhom1-eq}
\end{equation}
There is a continuous operator 
$K^{(m)}_1\colon \cs'(\Rnl)^m\to\cs'(\Rn)$, which maps $\cs'(\Rnl)^m$ into 
the domain of $\rho_{m,1}$ and is a right-inverse of $\rho_{m,1}$; and 
$K^{(m)}_{1}$ is furthermore
continuous with respect to the spaces in \eqref{rhom1-eq} for the specified
$s$. 
\end{cor}

\begin{cor}   \label{gjn-cor}
When
$s>(m-1)a_n+\frac{a_n}{p_n}+\sum_{k<n} 
(\frac{a_k}{\min(1,p_1,\dots,p_k)}-a_k)$
then $\rho_{m,n}$ is a bounded surjection
\begin{equation}
  \rho_{m,n}\colon \fspq(\Rn)\to \prod_{j=0}^{m-1}
  B^{s-ja_n-\frac{a_n}{p_n},a'}_{p',p_n}(\Rnl).
  \label{rhomn-eq}
\end{equation}
There is a continuous operator 
$K^{(m)}_n\colon \cs'(\Rnl)^m\to\cs'(\Rn)$, which maps $\cs'(\Rnl)^m$ into 
the domain of $\rho_{m,n}$ and is a right-inverse of $\rho_{m,n}$; and 
$K^{(m)}_{n}$ is furthermore
continuous with respect to the spaces in \eqref{rhomn-eq} for the specified
$s$. 
\end{cor}

\subsection{Remarks on the borderlines}
As illustrated in Figure~\ref{borderline_1-fig}, the mixed-norm spaces
$\fspq$ give borderline phenomena differing a
good deal from the well-known isotropic, unmixed $L_p$-theory (we take $\vec
a=(1,\dots,1)$ for simplicity): 
as a similarity $q$ plays no role, so we take $q=2$; then the spaces reduce 
to Sobolev spaces $H^s_{\vec p}=F^s_{\vec p,2}$ when $1<p_k<\infty$ 
for all $k$. Moreover, beginning with $\gamma_{0,1}$,
it is by \eqref{spq-cnd} of Theorem~\ref{main1-thm}
necessary that $s\ge 1/p_1$,
with $s=1/p_1$ being possible only for
$p_1\le1$. This requires in addition that
\begin{equation}
  \sum_{k>1}(\frac{1}{p_k}-1)_+=0,
\end{equation}
hence $p_k\ge1$ for all $k\ge2$.
However, $p_1\le1$ excludes the identification with a Sobolev space
(but every $u$ in $F^s_{\vec p,2}(\Rn)$ is then at least a  
continuous function of $x_1$ valued in the Banach space $L_{p''}(\Rnl)$). 

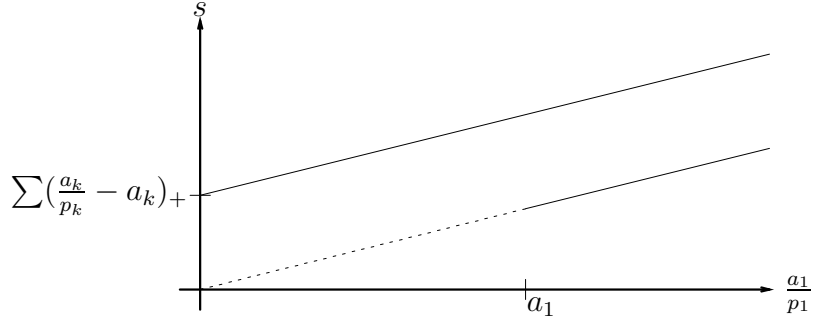
\begin{figure}[htbp]
\setlength{\unitlength}{0.00035in}
\begingroup\makeatletter\ifx\SetFigFont\undefined%
\gdef\SetFigFont#1#2#3#4#5{%
  \reset@font\fontsize{#1}{#2pt}%
  \fontfamily{#3}\fontseries{#4}\fontshape{#5}%
  \selectfont}%
\fi\endgroup%
{\renewcommand{\dashlinestretch}{30}
\hspace{1550\unitlength}
\begin{picture}(8800,5000)(-500,-10)
\path(8712,2412)(5112,1512)
\dashline{60.000}(5112,1512)(312,312)
\thicklines
\path(12,312)(8712,312)
\path(8592.000,282.000)(8712.000,312.000)(8592.000,342.000)
\put(8750,312){\makebox(0,0)[lc]{$\phantom{\Big|}\frac{a_1}{p_1}$}}
\path(312,12)(312,4312)
\path(342.000,4192.000)(312.000,4312.000)(282.000,4192.000)
\put(312,4380){\makebox(0,0)[cb]{$s$}}
\thinlines
\path(5112,462)(5112,162)
\put(5112,50){\makebox(0,0)[lc]{$\textstyle{a_1}$}}
\path(462,1712)(162,1712)
\put(150,1712){\makebox(0,0)[rc]{$\sum(\frac{a_k}{p_k}-a_k)_+$}}
\path(312,1712)(8712,3812)  
\end{picture}
}

\caption{The $\gamma_{0,1}$-borderlines 
for $s$, for different values of $p''$;
dashes indicate that $s$ must be strictly larger 
than at the borderline}
  \label{borderline_1-fig}
\end{figure}

When $\sum_{k>1}(\tfrac1{p_k}-1)_+>0$, i.e.\ at least one $p_k<1$
there is a marked difference to the
non-mixed case because the borderline is displaced upwards,
cf.~Figure~\ref{borderline_1-fig}.  
This is not unnatural, though, since there is a Sobolev
embedding, with $r_k=\max(1,p_k)$ for $k>1$, 
\begin{equation}
  F^s_{\vec p,2}(\Rn)\imb F^{\tfrac{1}{p_1}}_{(p_1,r''),2}(\Rn)
  \quad\text{for}\quad s=\tfrac{1}{p_1}+\sum_{k>1}(\tfrac{1}{p_k}-1)_+,
  \label{FF-eq}
\end{equation}
where the last space is located at the borderline for the Banach space case.
For $p_1\le1$ it is therefore clear that $\gamma_{0,1}$ is defined on 
$F^s_{\vec p,2}$, 
whereas for $p_1>1$ this might look contradictory.
But the meaning of Theorem~\ref{main1-thm} is that the subspace 
to the left in \eqref{FF-eq} is \emph{barely} small
enough to be in the domain of $\gamma_{0,1}$, even for $p_1>1$ 
(cf.\ the proof, where \eqref{FF-eq} 
is sharpened by a precise application of the vector-valued Nikol$'$skij 
inequality, cf.~\eqref{vNPP-ineq} below,
that allows a decisive shift to a sum exponent $q\le1$).

\subsection{The working definition of the trace}
For an overview of the methods, it is noted that we work with a 
quasi-homogeneous Littlewood--Paley decomposition $1=\sum_{j=0}^\infty
\Phi_j$ such that, for $j\ge1$,
\begin{equation}
\xi\in\supp\Phi_j\implies 2^{j-1}\le |\xi|_{\vec a}\le 2^{j+1}.
\end{equation}
Hereby $|\cdot|_{\vec a}$ stands for a quasi-homogeneous distance function,
with level sets given by $n$-dimensional ellipsoids of varying 
eccentricity; cf.\ Section~\ref{notation-ssect} for details. 

Decomposing $u=\sum \Phi_j(D)u$ there is an obvious candidate for the trace, 
say $\gamma_{0,1}$, for since the $\Phi_j(D)u$ are $C^\infty$-functions by 
the Paley--Wiener--Schwartz theorem, one can set
\begin{equation}
  \tilde\gamma_{0,1}u= \sum_{j=0}^\infty \Phi_j(D)u\bigm|_{x_1=0}.
  \label{work-defn}
\end{equation}
We adopt this as a working definition for $\gamma_{0,1}$. In fact, the proof
of \eqref{spq-cnd}$\implies$\eqref{cd'-cnd} in Theorem~\ref{main1-thm} 
shows that under the condition \eqref{spq-cnd}, the series
in \eqref{work-defn} converges in $L_{r''}$. 
But as the value $x_1=0$ does not play a special role,
a further argument yields
$\fspq(\Rn)\imb C_{\op{b}}(\R,L_{r''}(\Rnl))$.
The argument also shows that 
$\tilde\gamma_{0,1}$ is a map $\fspq(\Rn)\to\cd'(\Rnl)$
that is a 
\emph{restriction} of the distributional trace $\gamma_{0,1}$. 

Similar remarks apply to the outer trace $\gamma_{0,n}$.

\begin{rem}
Nikol$'$skij \cite{Nik75} assigned a trace on e.g.\ $\{x_n=0\}$ to any 
$f(x',x_n)$ behaving as an $L_{p'}$-function in $x'$ and
depending continously (near $x_n=0$) on the parameter $x_n$, 
i.e.\ to any $f$ in $C(\R,L_{p'}(\Rnl))$. The trace is of course defined on
the larger space $C(\R,\cd'(\Rnl))$, but by Theorems~\ref{main1-thm} and
\ref{mainn-thm}, the $\fspq$ that admit traces are regular enough to 
fulfil Nikol$'$skij's requirement, at least when the components of 
$r'$ or $r''$ are equal.
\end{rem}

\subsection{Anisotropic Sobolev spaces}

For comparison's sake,
we collect the relation to the anisotropic counterparts of the well-known 
Bessel potential and Sobolev spaces. 
For brevity, e.g.\ $1<\vec p<\infty $
means that $1<p_k<\infty $ for all $k=1,\dots ,n$.

\begin{prop}\label{Littlewood}
Let $1 < \vec{p} < \infty$ and $s \in \R$ be arbitrary.
\begin{itemize}
  \item[{\rm (i)}] Then $F^{s,\vec{a}}_{\vec{p},2} (\Rn) = \hsp (\Rn)$  
where $H^{s,\vec a}_{\vec p}$ consists
of the $u\in\cs'(\Rn)$ for which
\be
 \Norm{\cfi \big[ 
 (1+|\xi|_{\vec a}^2 )^{s/2} \, \cf u \big](\, \cdot\,) 
  }{\lpvec(\Rn)}
< \infty. 
\ee
  \item[{\rm (ii)}] When $m_k=\tfrac{s}{a_k}\in \N_0$  for each  
$k=1, \ldots , n$, 
then $F^{s,\vec{a}}_{\vec{p},2} (\Rn) = \wmp (\Rn)$ for 
$\vec m=(m_1,\dots,m_n)$, 
where $\wmp$ consists of the $u\in\cs'(\Rn)$ such that
\be
  \Norm{u}{\lpvec(\Rn)} +\sum_{i=1}^n 
 \Norm{\frac{\partial^{m_i} u}{\partial x_i^{m_i}}}{\lpvec(\Rn)}
  < \infty .
\ee
\end{itemize}
In both cases the norms are equivalent to that of $\fspq$.
\end{prop}

The essential part of this result goes back to Lizorkin \cite{Liz70},
who introduced and discussed the above spaces.

\bigskip

Conversely to Proposition~\ref{Littlewood},
one often needs to identify a given Sobolev space $\wmp$ with a
Lizorkin--Triebel space. While this can be done in many ways, 
we first recall
the convention, preferred in the Russian school, e.g.~\cite{BIN78,Liz70},
of taking the smoothness $s$ as the harmonic mean of the given
orders,
\begin{equation}
  \frac{1}{s}=\frac{1}{n}(\frac{1}{m_1}+\dots+\frac{1}{m_n}).
  \label{Lconv-eq}
\end{equation}
Then, by setting $a_k=s/m_k$ for $k=1,\dots,n$, 
Proposition~\ref{Littlewood} clearly gives
\begin{equation}
  \wmp(\Rn)=F^{s,\vec a}_{\vec p,2}(\Rn),
  \quad\text{and}\quad a_1+\dots+a_n=|\vec a|=n.
  \label{Wan-eq}
\end{equation}
This yields the following trace results for Sobolev spaces.

\begin{prop}
Let $\vec m=(m_1,\dots,m_n)\in \N_0^n$ and
$1<p_k<\infty$ for $k=1,\dots,n$, and define $s$ by \eqref{Lconv-eq}
and $a_k=s/m_k$ for all $k$. Then there are bounded surjections
\begin{align}
  \gamma_{0,1}&\colon \wmp(\Rn)\to F^{s-\frac{a_1}{p_1},a''}_{p'',p_1}(\Rnl)
  \quad\text{for}\quad m_1>\frac1{p_1},
\\
  \gamma_{0,n}&\colon \wmp(\Rn)\to B^{s-\frac{a_n}{p_n},a'}_{p',p_n}(\Rnl)
  \quad\text{for}\quad m_n>\frac1{p_n}.
\end{align}
\end{prop}
Note that substitution of e.g.\  $a_1=s/m_1$ entails
$s-\tfrac{a_1}{p_1}=s(1-\tfrac{1}{m_1p_1})$, where
the last expression is used by some authors.

However, as an alternative to \eqref{Lconv-eq}-\eqref{Wan-eq}, 
there is also an identification 
\begin{equation}
 \wmp(\Rn)= F^{s,\vec a}_{\vec p,2}(\Rn)  
  \quad\text{with}\quad 
  s=\max(m_1,\dots,m_n).
  \label{smax-eq}
\end{equation} 
Indeed,
it is verified in Lemma~\ref{Flambda-lem} below that 
$\fspq=F^{\lambda s,\lambda\vec a}_{\vec p,q}$ with equivalent quasi-norms,
for every $\lambda>0$.
So \eqref{smax-eq} follows from \eqref{Wan-eq} for 
$\lambda=\tfrac{1}{n}(\tfrac{1}{m_1}+\dots+\tfrac{1}{m_n})
\max(m_1,\dots,m_n)$.
Then the weigths in \eqref{smax-eq} fulfill
\begin{equation}
  a_k=\frac{1}{m_k}\max(m_1,\dots,m_n) \quad\text{for } k=1,\dots,n;
\qquad
  \min(a_1,\dots,a_n)=1.
  \label{sak-eq}
\end{equation}
In particular this gives the normalisation $\min(a_1,\dots,a_n)=1$, 
instead of $|\vec a|=n$.

Another virtue of \eqref{smax-eq}--\eqref{sak-eq} 
is that every $m_k\in [0,s]$. 
Moreover, in \eqref{WF-eq} the space
$W^{2,1}_{\vec p}(\Rn\times\R)$ stands for $W^{2,\dots,2,1}_{\vec
p}(\Rn\times\R)$, so \eqref{sak-eq} clearly gives 
$\vec a=(1,\dots,1,2)$; cf.~\eqref{WF-eq}.

We prefer to adopt the convention that $\min(a_1,\dots,a_n)=1$ in the proofs, 
since it makes
some estimates simpler and gives direct reference to e.g.\ 
\cite{Y1,JJ96ell,FaJoSi00,JoSi}.

\begin{rem}[related work]
Traces of mixed norm Sobolev spaces $\wmp$ were covered by 
Bugrov \cite{Bug71}. 
In a series of papers \cite{Ber84,Ber85,Ber87I,Ber87II} Berkolaiko proved
Theorems~\ref{main1-thm'}--\ref{mainn-thm'} with all $p_k$ and $q$ 
in $\,]1,\infty[\,$. He also obtained the condition
$s>\frac{a_k}{p_k}$ for these cases 
(whereas corrections for $0<p_k<1$ can be found in the present paper).

Moreover, Berkolaiko showed that for $k=2,\dots,n-1$ the ranges of 
$\gamma_{0,k}$ are given neither by Besov nor Lizorkin--Triebel spaces; 
instead the relevant norms will have the discrete $\ell_q$-norm 
`replacing' that of $L_{p_k}$ (as is shown here for $k=1$ and $k=n$). 
We have refrained from going into this,
since $\gamma_{0,1}$ and $\gamma_{0,n}$ should suffice for 
most parabolic problems.

It was seemingly first realised by Weidemaier \cite{Wei98}
that it is relevant for the fine theory of parabolic problems 
to have Lizorkin--Triebel spaces as trace spaces.
Among the other works on this application we can mention
\cite{DeHiPr,Wei02,Wei05}.
\end{rem}


\section{Lizorkin--Triebel spaces $\fspq$ based on mixed norms}
  \label{mixd-sect}

\subsection{Notation and preliminaries}
  \label{notation-ssect}

For a given $\vec{p}=(p_1,\dots, p_n)$ with $p_k \in
\,]0,\infty]$, $k=1, \ldots , n$,
we denote by $L_{\vec{p}}(\Rn)$ the set of all equivalence classes of
measurable functions $u: \Rn \to \C$ such that 
\be
\Norm{u}{L_{\vec{p}}(\Rn)}:= 
\bigg(\int_{\R}  \ldots \bigg(  \int_{\R} 
\bigg( \int_{\R} |u(x_1,\ldots ,x_n)|^{p_1} \,dx_1
\bigg)^{\frac{p_2}{p_1}} \, dx_2 \bigg)^{\frac{p_3}{p_2}} 
\ldots  
\,dx_n \bigg)^{\frac1{p_n}} 
\ee
is finite (modification if some of the $p_i$ are
equal to $\infty$).
With this quasi-norm $L_{\vec{p}}(\Rn)$ is complete,
and a Banach space if $\min (p_1, \ldots , p_n) \ge 1$.
Furthermore, for $0 < q \le \infty$, we shall use the abbreviation
$\lpvec (\ell_q)(\Rn)$ for the set of all sequences 
$(u_k)_{k\in\N_0}$, also written as $\{u_k\}_{k=0}^\infty$, of measurable
functions $u_k\colon \Rn \to \C$ such that (with $\sup_k$ for $q=\infty$)
\begin{equation}
 \Norm{\{u_k\}_{k=0}^\infty }{ \lpvec (\ell_q)(\Rn)} :=
\bigg\| \, \Big(\sum_{k=0}^\infty |u_k|^q
\Big)^{1/q}\bigg|\lpvec (\Rn)\bigg\| < \infty .
\end{equation}
For brevity 
$\Norm{u_k}{\lpvec (\ell_q)}$ may replace 
$\Norm{ \{u_k\}_{k=0}^\infty}{\lpvec (\ell_q)(\Rn)}$. 
If $\max (p_1, \ldots, p_n, q)<\infty$, then
sequences $\{u_k\}_{k=0}^\infty$ from
$C_0^\infty$ are dense in $\lpvec (\ell_q)(\Rn)$. 
$L_{\vec p}$ was studied by Benedek and Panzone
\cite{BePa61}.

In general we adopt standard notation from distribution theory.
E.g.\ $\cd'(\Rn)$ stands for the space of distributions on $\Rn$, while
$\cs'(\Rn)$ is the subspace of tempered distributions.
The Fourier transformation is denoted by $\cf u=\hat u$, 
where $\cf u(\xi)=\int_{\Rn} e^{-\im x\cdot\xi}u(x)\,dx$ for 
$u\in \cs(\Rn)$ with $\cs(\Rn)$ being the Schwartz space of rapidly
decreasing $C^\infty$-functions on $\Rn$.

On $\Rn$ we use an anisotropic distance function $|\cdot|_{\vec a}$ 
of a quasi-homogeneous type, when $\vec a=(a_1,\dots,a_n)$ is fixed in
$\,]0,\infty[\,^n$ (cf.\ Remark~\ref{veca-rem}).
First $\vec a$ is used for the quasi-homogeneous dilation 
$t^{\vec a}x:=(t^{a_1}x_1,\dots,t^{a_n}x_n)$ for $t\ge0$, and 
$t^{s\vec a}x:=(t^s)^{\vec a}x$ for $s\in \R$, whence
$t^{-\vec a}x=(t^{-1})^{\vec a}x$. Then
$|x|_{\vec a}$ is the unique $t>0$ such
that $t^{-\vec a}x\in S^{n-1}$ ($|0|_{\vec a}=0$), i.e.\
\begin{equation}
  \tfrac{x_1^2}{t^{2a_1}}+\dots+  \tfrac{x_n^2}{t^{2a_n}}=1.
  \label{tax-eq}
\end{equation}
It is seen directly that $|t^{\vec a}x|_{\vec a}=t|x|_{\vec a}$, 
so $|\cdot|_{\vec a}$ is not a norm for $\vec a\ne(1,\dots,1)$, but one has
\begin{gather}
  |x+y|_{\vec a}\le |x|_{\vec a}+|y|_{\vec a}.
  \label{atriangle-ineq}
\\
  \max(|x_1|^{1/a_1},\dots,|x_n|^{1/a_n})
  \le |x|_{\vec a}\le |x_1|^{1/a_1}+\dots+|x_n|^{1/a_n}.
  \label{ax-ineq}
\end{gather}
We set
$B_{\vec a}(x,R):= \{\,y \mid |x-y|_{\vec a} \le R\,\}$. A review of
$|\cdot|_{\vec a}$ can be found in \cite{JoSi,Y1}.

Along with $|\cdot|_{\vec a}$, a quasi-homogeneous
Littlewood--Paley decomposition $1=\sum\Phi_j$ will be chosen 
as follows:
based on some $\psi \in C^\infty(\R)$ such that
$0 \le \psi (t) \le 1$ for all $t$,
$\psi (t) =1$ if $\le 11/10$, and $\psi (t) =0 $ if $t>13/10$,
we set $\Psi_j(\xi):=\psi(2^{-j}|\xi|_{\vec a})$ for $j\in\N_0$ 
($\Psi_j\equiv0$ for $j<0$)
so that $\Phi_j:=\Psi_j-\Psi_{j-1}$ gives 
$1=\sum_{j=0}^\infty\Phi_j(\xi)$ for all $\xi\in\Rn$.
Clearly  
\begin{equation}
  \supp\Phi_j\subset\{\, \xi\mid \tfrac{11}{20}2^j\le |\xi|_{\vec a}
  \le \tfrac{13}{10}2^j\,\}.
\end{equation}
This choice is indicated by the uppercase letters $\Psi$, $\Phi$ throughout.
Whenever $1<p_k<\infty$ for $k=1,\dots,n$, then a
Littlewood--Paley inequality holds for all $u\in  L_{\vec p}(\Rn)$:
\begin{equation}
  c_1\Norm{u}{L_{\vec p}}\le
  \Norm{ (\sum_{j=0}^\infty |\cfi[\Phi_j\cf u]|^2)^{\frac{1}{2}}}{L_{\vec p}}
  \le c_2\Norm{u}{L_{\vec p}}.
  \label{lwp-ineq}
\end{equation}
In fact the right-hand side inequality follows directly from a theorem of
Kr\'ee \cite[Th.~4]{Kr67}; then the inequality to the left is obtained from the
completeness of $L_{\vec p}$ and duality (cf.\ a similar proof in 
\cite[Prop.~3.3]{Y86}).


\subsection{Lizorkin--Triebel spaces with mixed norms}


Let  $\Phi_j$, $j \in \N_0$, be our anisotropic dyadic decomposition of unity.

\begin{defn}   \label{fspq-defn}
Let $0<p_1, \ldots , p_n < \infty$, $s\in \R$, and $0 < q \le \infty$.
Then the quasi-homogeneous \emph{mixed-norm} Lizorkin--Triebel space 
$\fspq (\Rn)$ is the set of $u\in\cs'(\Rn)$ such that
\begin{equation*}
\Norm{u}{\fspq} := 
\bigg\| \bigg(\sum_{j=0}^\infty 2^{jsq} |\cfi [\Phi_j 
\cf u] (\cdot)|^q \bigg)^{\frac1q} \bigg| L_{\vec{p}}(\Rn)
\bigg\| < \infty.
  \end{equation*}
\end{defn}

The $\fspq (\Rn)$
are quasi-Banach spaces, and Banach spaces if $p_1,\ldots, p_n,q$ all 
belong to $[1,\infty]$. Instead of the quasi-triangle inequality, 
it is useful that for all $u$, $v \in \fspq (\Rn)$ the number 
$\tau=\min(1,p_1,\dots, p_n,q)$ gives rise to the estimate
\begin{equation}
 \norm{u+v}{\fspq}^\tau \le 
 \norm{u}{\fspq}^\tau +\norm{v}{\fspq}^\tau . 
 \label{tau-eq}
\end{equation}
Up to equivalent quasi-norms, the spaces 
$\fspq (\Rn)$ do not depend on the chosen anisotropic 
dyadic decomposition of unity 
For brevity $\cal F^{-1}(\Phi_j\cal Fu)$ is often written as
$\Phi_j(D)u$.

We shall also need the corresponding Besov spaces. 
They have properties like the above-mentioned for the $\fspq$, 
so we just give the definition. 

\begin{defn}
For $0<p_1, \ldots , p_n,q \le \infty$ and $s\in \R$
the quasi-homogeneous \emph{mixed-norm} Besov space 
$\bspq (\Rn)$ consists of all $u\in\cs'(\Rn)$ such that
\begin{equation*}
  \Norm{u}{\bspq} := 
  \bigg(\sum_{j=0}^\infty 2^{jsq} \norm{\cfi (\Phi_j  
      \cf u)}{\lpvec (\Rn)}^q \bigg)^{\tfrac 1q} 
   < \infty.
\end{equation*}
\end{defn}

\begin{prop}   \label{FBtrans-prop}
$\fspq(\Rn)$ is translation invariant; and
for $q<\infty$ and every $u\in \fspq(\Rn)$, the translations 
$\tau_h u:=u(\cdot-h) \to u$ in
$\fspq$ for $h\to 0$. Analogously $u\in \bspq(\Rn)$ implies $\tau_hu \in
\bspq(\Rn)$, with $\tau_hu\to u$ when $q$ and all $p_k$ are finite.
\end{prop}
\begin{proof}
Since $\Phi_j(D)\tau_h=\tau_h\Phi_j(D)$, the norm of $\fspq$ 
is translation invariant, as that of $L_{\vec p}(\Rn)$ is so. Hence both 
$u$, $\tau_hu$ may be approximated in $\fspq$ to within an $\varepsilon$, 
by choosing a suitable $\psi\in\cs$, when $q<\infty$. 
And $\norm{\tau_h\psi-\psi}{\fspq}\to0$ 
for $h\to0$, because $\tau_h\psi\to\psi$ in $\cs(\Rn)$ and 
the injection $\cs\imb\fspq$ is continuous. (Clearly $\bspq$ can replace
$\fspq$ here.)
\end{proof}

\begin{rem}
For $\vec a=(1,1,\ldots ,1)$ these spaces fit into the general scheme 
developed by Hedberg and Netrusov, cf.~\cite{HeNe}. 
So in the isotropic situation we have a lot of properties at hand for
these classes like characterization by atoms, characterization by 
oscillations (local approximation by polynomials) and characterization
by differences. We envisage
that most of the material presented there has a counterpart for the
anisotropic spaces. 
\end{rem}


\subsection{Embedding results}


For a continuous linear injection of $X$ into $Y$ we throughout write 
$X\imb Y$. A proof of the next result is given further below.

\begin{lem}   \label{SFS-lem}
There are continuous embeddings
\begin{gather}
 \cs (\Rn) \hookrightarrow \fspq (\Rn) \hookrightarrow \cs' (\Rn).  
  \label{SFS'-eq}
    \\
\cs (\Rn) \hookrightarrow \bspq (\Rn) \hookrightarrow \cs' (\Rn).
  \label{SBS'-eq}  
\end{gather}
$\cs(\Rn)$ is dense in $\fspq(\Rn)$ for $q<\infty$, and 
dense in $\bspq(\Rn)$ for $q,p_1,\dots,p_n<\infty$.
\end{lem}

The definitions at once give 
part \mbox{(i)} of the next result;
and \mbox{(iii)} follows from \mbox{(ii)}, that holds by 
Minkowski's inequality.

\begin{lem}\label{elementary}
When $p_k<\infty$ holds for all $k$ in the $F$-spaces one has:
\begin{rmlist}
\item
For $s' < s$ and $q$, $q'\in\,]0,\infty]$,
\begin{equation}
  \fspq (\Rn)\imb F^{s',\vec{a}}_{\vec{p},q'} (\Rn);
\qquad
  \bspq (\Rn)\imb B^{s',\vec{a}}_{\vec{p},q'} (\Rn).  
\end{equation}
\item 
For $r_1 \le \min (p_1,\ldots , p_n, q)$ and 
$\max (p_1,\ldots, p_n,q)\le r_2$,
\be
\Big(\sum_{j=0}^\infty \norm{u_j}{
  \lpvec (\Rn)}^{r_2}\Big)^{\frac1{r_2}}
 \le \Norm{u_j}{\lpvec (\ell_q)(\Rn)}
 \le   
\Big(\sum_{j=0}^\infty \norm{u_j}{\lpvec(\Rn)}^{r_1}\Big)^{\frac1{r_1}},
\ee
for an arbitrary sequence $(u_j)$ of measurable functions.
\item
With $r_1$ and $r_2$ as in \mbox{\rm (ii)},
  \begin{equation}
    B^{s,\vec{a}}_{\vec{p},r_1} (\Rn) \imb
      \fspq (\R_n) \imb B^{s,\vec{a}}_{\vec{p},r_2} (\Rn).
  \end{equation}
  \end{rmlist}
\end{lem}
Let $\vec b=(b_1, \ldots , b_n) \in \Rn$ such that $b_k >0$, 
$k=1, \ldots ,n$. As a convenient notation we introduce the cube
\begin{equation}
  Q_{\vec b} := \Big\{ (x_1, \ldots x_n)\bigm| \: |x_k| \le b_k, 
  \: k=1, \ldots, n \Big\}  
  \label{Qb-eq}
\end{equation}
The symbol $x \cdot y$ refers to the scalar product of $x$, $y$ in $\Rn$.
For a vector $\vec{r}$ we shall as a convention set
\begin{equation}
  \tfrac{1}{\vec{r}} = \Big(\tfrac{1}{{r_1}}, 
  \tfrac{1}{{r_2}},\ldots,  \tfrac{1}{{r_n}}\Big).  
\end{equation}
In our proofs the 
vector-valued Nikol$'$skij inequality will play a major role.
This inequality concerns sequences $(f_j)$ in $\cs'(\Rn)$ 
that fulfill a \emph{geometric rectangle condition},
\begin{equation}
  \supp \cf f_j \subset [-AR_1^j,AR_1^j]\times\dots\times[-AR_n^j,AR_n^j].
  \label{vNPP-cnd}
\end{equation}
Here $A>0$ is a constant, while the fixed numbers $R_1$,\dots,$R_n>1$
define the rectangles.

\begin{thm}
  \label{vNPP-thm}
When $0<p_k\le r_k<\infty$ for $k=1,\dots,n$ and $\vec r\ne \vec p$, 
then there is
for $0<q\le\infty$ a number $c>0$ such that
\begin{equation}
  \Norm{(\sum_{j=0}^\infty |f_j(\cdot)|^q)^\frac1q}{L_{\vec r}}
  \le c
  \Norm{\sup_{j\in \N_0} (\prod_{k=1}^nR_k^{j(\frac 1{p_k}-\frac 1{r_k})}
          |f_j(\cdot)|)}{L_{\vec p}}
  \label{vNPP-ineq}
\end{equation}
for all sequences $(f_j)$ in $\cs'(\Rn)$ fulfilling \eqref{vNPP-cnd}.
\end{thm}

For the proof the reader is referred to \cite[Thm.~5]{JoSi}.
As noted there, this vector-valued Nikol$'$skij inequality at once 
gives Sobolev embeddings for the $\fspq$, 
where by virtue of \eqref{vNPP-ineq} it suffices to increase only 
a single component $p_k$ of $\vec p$:

\begin{cor} \label{emb5}
When $0 < p_k \le r_k< \infty$ for all $k$ and $\vec{r} \neq \vec{p}$, 
then 
\be
F^{s,\vec{a}}_{\vec{p},q_1} (\Rn) \hookrightarrow  
F^{t,\vec{a}}_{\vec{r},q_2} (\Rn) 
\ee
holds for 
$t =s - \vec{a} \cdot \big(\frac{1}{\vec{p}}-\frac{1}{\vec{r}}\big)$.
\end{cor}

The classical Nikol$'$skij inequality deals with a 
single function with compact spectrum.
This results by applying \eqref{vNPP-ineq} 
to a sequence with a single non-trivial element; then also
$r_k=\infty$ is allowed (cf.~\cite[Thm.~4]{JoSi}). This will, 
by the definition of $\bspq$, give 

\begin{cor}\label{besove}
Suppose
$0 < p_k \le r_k\le \infty$ for all $k$; $\vec r\ne\vec p$. Then
\be
B^{s,\vec{a}}_{\vec{p},q_1} (\Rn) \hookrightarrow  
B^{t,\vec{a}}_{\vec{r},q_2} (\Rn) 
\ee
holds if $t - \vec{a} \cdot \tfrac{1}{\vec{r}} 
< s - \vec{a} \cdot \tfrac{1}{\vec{p}}$, or if both
$t -\vec{a}\cdot\tfrac{1}{\vec{r}} = s-\vec{a}\cdot\tfrac{1}{\vec{p}}$
and $q_1 \le q_2$.
\end{cor}

By definition, every $u\in B^{0,\vec a}_{\vec\infty,1}$, 
has finite norm series in $L_\infty$,
whence $B^{0,\vec a}_{\vec\infty,1}(\Rn)\imb C_{\op{b}}(\Rn)$.
Therefore Lemma~\ref{elementary} 
and Corollary~\ref{besove} give
$F^{s,\vec{a}}_{\vec{p},q} (\Rn) \imb
B^{s- \vec{a}\cdot \frac{1}{\vec{p}},\vec{a}}_{\vec{\infty},\infty} 
(\Rn) $,
so
\begin{equation}
  \fspq(\Rn)\imb C_{\op{b}}(\Rn) \quad\text{for}\quad
  s>\vec a\cdot\frac 1{\vec p}.
\end{equation}

\begin{rem}
The embeddings and inequalities of this section have 
been extensively studied, in many versions, over several decades. 
It would be outside of our topic to recall this here, \cite{BIN96} or 
\cite{ScTr87} may be consulted as a general reference;
\cite{JoSi} has remarks on the development, 
as well as proofs pertaining to 
the anisotropic framework used here.
\end{rem}


\subsection{Maximal inequalities}


As usual we let $Mf$ denote the Hardy--Littlewood maximal function, 
defined for a locally integrable function on $\Rn$ by
\be
  Mf(x)=\sup_{r>0} \frac1{\op{meas}(B(0,r))}
  \int_{B(0,r)} |f(x+y)|,dy.
\ee
When the definition of $M$ is applied only in the variable $x_k$,
we shall via the splitting $x=(x',x_k,x'')$ use the abbreviation
\be
M_k u(x_1, \ldots ,x_n) :=
(M u(x', \cdot , x''))(x_k)
\ee
Using this, we can formulate an important inequality due to Bagby \cite{Bag75}.
Let $1 < p_n < \infty $,  
and let $1 < q,p_k \le \infty$ for $k<n$. 
Then there exists a constant $c$ such that every sequence in 
$L_{\vec p}(\ell_q)$ fulfils the inequality 
\be\label{bagby}
  \Norm{M_n u_j}{L_{\vec{p}}(\ell_q)(\Rn)}
\le c \Norm{u_j}{L_{\vec{p}}(\ell_q) (\Rn)}.
\ee
It is well known that this allows the iterated maximal function
$M_n(\dots M_2(M_1 f)\dots)(x)$ to be estimated
in the mixed-norm space $L_{\vec p}$.

However, we shall also use the
maximal function of Peetre--Fefferman--Stein type,
\be \label{u*-id}
  u^* (\vec{r}, \vec b ; x) =
  \sup_{z \in \Rn}
  \frac{|u(x-z)|}{(1+ |b_1z_1|^{1/r_1}) \ldots 
    (1+|b_nz_n|^{1/r_n})} .
\ee
In our cases the function $u$ will have compact spectrum, and then $u^*$ is 
majorised by the iterated Hardy--Littlewood maximal function. 
As a first step one has the next result.
 
\begin{prop}\label{max*M-prop}
Suppose $0 < \vec{r} < \infty$,
and consider a cube $Q_{\vec b}$ as in \eqref{Qb-eq}.
Then there exist a constant $c>0$ such that
\be
\label{max1}
  \sup_{z \in \Rn}
  \frac{|u(x-z)|}{(1+ |z_1|^{1/r_1}) \ldots 
    (1+|z_n|^{1/r_n})} 
\le   c \big(M_n (\ldots M_2(
  M_1 |u|^{r_1})^{r_2/r_1} \ldots  )^{r_{n}/r_{n-1}} \big)^{1/r_n}(x)
\ee
holds whenever $\supp \cf u \subset Q_{\vec b}$ and $u\in L_{\vec p}(\Rn)$ 
for $0< p_k<\infty$ for all $k$. 
\end{prop}

The proof given in \cite[Thm.~1.6.4]{ScTr87} for $n=2$ is easily 
extended to arbitrary dimensions.
Combined with a dilation,
Proposition~\ref{max*M-prop} gives, as in \cite[1.10.2]{ScTr87},
a vector-valued estimate for the 
Fefferman--Stein maximal function, which will be central to
our trace estimates in Section~\ref{prfs-sect}:

\begin{prop}\label{maximalf}
Let $0 < \vec{p} < \infty$, $0 < q \le \infty$, and suppose
every component of $\vec r$ satisfies
\be
  0 < r_k < \min (p_1,\ldots,p_k, q).
\ee
Then there exists a  $c>0$ such that,
whenever  $(\vec b^j)$ is a sequence in 
$]0,\infty[\,^{\raise 2pt\hbox{$\scriptstyle n$}}$,
\be \label{maximalf-ineq} 
  \Norm{u_j^* (\vec{r},\vec b^j, \cdot)}{\lpvec (\ell_q)(\Rn)}
 \le   c 
  \Norm{u_j}{\lpvec (\ell_q)(\Rn)}
\ee
holds for all sequences $(u_j)$ in $L_{\vec{p}} (\ell_q)(\Rn)$ 
such that $\supp \cf u_j \subset Q_{\vec b^j}$ for all $j\in\N_0$.
\end{prop}

\begin{proof}
We apply Proposition~\ref{max*M-prop} to
\be \label{gj-eq}
  g_j (x)= u_j (x_1/b_1^j, \ldots , x_n / b_n^j).
\ee
Obviously
$\supp \cf g_j \subset Q_{(1, \ldots , 1)}$ for every $j$, and we have
\be
  g^*_j(x)
 \le   c_1  \big(M_n (\ldots M_2(
            M_1 |g_j|^{r_1})^{r_2/r_1} \ldots  )^{r_{n}/r_{n-1}}
      \big)^{1/r_n}(x),
\ee
where $c_1$ is independent of $j$. Now \eqref{gj-eq} and 
$x= (b_1^jy_1, \ldots , b_n^j y_n ) $ give
\be
  g^*_j(x)= 
   \sup_{z \in \Rn} \, \frac{|g_j(b_1^j y_1 - b_1^j z_1,\ldots,
   b_n^j y_n - b_n^j z_n) |}{(1+ |b_1^j z_1|^{1/r_1})\ldots  
  (1+ | b_1^j z_n |^{1/r_n})}
  = u^*_j(\vec r, \vec b; y).
\ee
Moreover, $M$ commutes with dilation, i.e.\ 
$Mf(\delta x)=Mf(\delta\cdot)(x)$, so
\begin{multline}
 \big(M_n (\ldots M_2(
  M_1 |g_j|^{r_1})^{\frac{r_2}{r_1}} \ldots  )^{\frac{r_n}{r_{n-1}}}
  \big)^{\frac1{r_n}}(y)(b_1^jy_1, \ldots , b_n^j y_n ) 
\\
 = 
  \big(M_n (\ldots M_2(
  M_1 |g_j (b^j_1 \, \cdot, \ldots, b_n^j\cdot)
      |^{r_1})^{\frac{r_2}{r_1}} \ldots  )^{\frac{r_n}{r_{n-1}}}
  \big)^{\frac1{r_n}}(y).
\end{multline}
In view of \eqref{gj-eq} this means that
\be
  u^*_j(\vec r, \vec b^j;y)
 \le   c_1  \big(M_n (\ldots M_2(M_1 |u_j
  |^{r_1})^{\frac{r_2}{r_1}}\ldots)^{\frac{r_{n}}{r_{n-1}}}
  \big)^{\frac1{r_n}}(y).
\ee
Applying Bagby's inequality \eqref{bagby} 
to $L_{(p_1/r_n, \, \ldots \, , p_n/r_n)} (\ell_{q/r_n})$
(using that all exponents belong to $\,]1,\infty[\,$,
by the restriction on $r_n$), this gives
\be
   \Norm{u_j^* (\vec{r}, \vec b^j;\cdot)}{\lpvec (\ell_q)(\Rn)}
 \le   c_2 
   \Norm{
      \big(M_{n-1}\ldots 
      M_2(M_1 |u_j|^{r_1})^{\frac{r_2}{r_1}} \ldots  \big)^{\frac1{r_{n-1}}}
   }
   {\lpvec (\ell_q) (\Rn)}.  
\ee
By freezing $x_n$, 
Bagby's inequality \eqref{bagby} applies to 
$L_{(p_1/r_{n-1}, \ldots , p_{n-1}/r_{n-1})} 
(\ell_{q/r_{n-1}}) (\R^{n-1})$. 
And by reiterating this, the statement follows.
\end{proof}


\subsection{Marschall's inequality}


Inspired by Marschall's paper \cite{Mar96}, we shall give a version of  
his pointwise estimate of pseudo-differential operators
$b(x,D)$, that is suitable for the mixed norm spaces.

In Marschall's inequality the symbol is estimated
via the norm of a homogeneous Besov space $\dot B^{s,\vec a}_{p,q}(\Rn)$. 
To recall the definition of  the norm,
we need a dyadic partition of unity,
$1 = \sum_{k=-\infty}^\infty \phi_k$ on $\Rn \setminus \{0\}$.
This can be obtained from the previously introduced functions,
by setting $\phi_j=\psi(2^{-j}|\cdot|_{\vec a})-\psi(2^{1-j}|\cdot|_{\vec a})$
for all $j\in\Z$. With this,
\be \label{e7}
  \supp \phi_k  \subset  B_{\vec{a}}(0,2^{k+1})\setminus 
  B_{\vec{a}}(0,2^{k-1}) \subset Q_{2^{(k+1)\vec a}(1,\dots,1)}.
\ee
Using $(\phi_j)_{j\in\Z}$, 
the norm $\norm{\cdot}{\dot B^{s,\vec a}_{p,q}}$ of 
$\dot B^{s,\vec a}_{p,q}(\Rn)$ is defined in analogy with that $\bspq$, 
simply by summing over $\Z$. 
It follows straightforwardly that 
\be \label{Bhomo-eq}
  \Norm{f(2^{k\vec a}\cdot)}{\dot B^{s,\vec a}_{p,q}}
  =
  2^{k(s-\frac{|\vec a|}{p})} \norm{f}{\dot B^{s,\vec a}_{p,q}},
  \qquad k\in \Z.
\ee
This scaling relation is
the important property we need from this tool.

For the anisotropic weights, i.e.\ $\vec a$, the length is denoted by
$|\vec a|=a_1+\dots+a_n$ for simplicity's sake.

\begin{prop}\label{Marschall}
Let a symbol $b \in C_0^\infty (\Rn)$ 
and a function $u \in C^\infty(\Rn)$ be given such that,
for $A>0$ and $R \ge 1$,
\be
  \supp \cf u \subset B_{\vec{a}}(0,AR) 
  \quad \text{and}\quad
  \supp  b \subset B_{\vec{a}} (0,A) 
\ee
When $\vec{t} = (t_1, \ldots \, , t_n)$ 
satisfies $0<t_k\le 1$ for all $k$, 
then there exists $c>0$ such that
the following inequality holds for all $x\in\Rn$,
with $d:= \min (1,t_1,\dots, t_n)$, 
\be
 |b(D)u(x)|
  \le c (RA)^{\vec{a}\cdot\frac{1}{\vec{t}}-|\vec a|} \, 
  \Norm{b}{\dot{B}^{\vec{a}\cdot\frac{1}{\vec{t}}, {\vec{a}}}_{1,d}} 
  \big(M_n ( \ldots \, (M_1 |u|^{t_1})^{t_2/t_1} \ldots 
  )^{t_n/t_{n-1}}\big)^{1/t_n}(x).
\ee
Here $c$ can be taken as a function of $\vec a$ and $\vec t$ only.
\end{prop}

\begin{proof}
Since convolutions in $\cs*\cs'$ are mapped to products by the 
Fourier transformation,
\be
  b(D) u(x)=\cfi(b\cf u)(x) = 
  \int \cfi b (x-y)u(y)\, dy.
\ee
With $x$ fixed, $y \mapsto \cfi b (x-y)u(y)$ has, 
by the triangle inequality for $|\cdot|_{\vec a}$, its  spectrum in 
\be
  B_{\vec{a}}(0, A) + B_{\vec{a}}(0,RA) \subset B_{\vec{a}}(0, (R+1)A).
\ee
Therefore the Nikol$'$skij inequality \eqref{vNPP-ineq} and 
an $L_{\vec p}\,$-version of \eqref{tau-eq} yields 
\begin{equation} \label{e8}
  \begin{split}
  |b(D)u(x)| & \le   \int |\cfi b (x-y)u(y)|\, dy 
\\
  & \le  c_1  
  (RA)^{\vec{a}\cdot\frac{1}{\vec{t}}-|\vec a|} 
  \Norm{\cfi b (x-\cdot)u}{L_{\vec{t}}} 
\\
  &\le   c_1  
  (RA)^{(\vec{a}\cdot\frac{1}{\vec{t}}-|\vec a|)}
  \Big(\sum_{k\in \Z}  \Norm{\phi_k(x-\cdot) 
  \cfi b (x-\cdot)u}{L_{\vec{t}}}^d \Big)^{1/d}.
  \end{split}
\end{equation}
In this inequality it suffices for the $L_{\vec t}$-norm, by \eqref{e7}, to
integrate over a cube on the right-hand side, and by the obvious estimate 
$\sup_y|\phi_k (y)\cfi b (y)|  
  \le  \int \big|\cfi_{y \to \eta} 
  (\phi_k \cfi b ) \big|  \, d\eta=:  b_{k}$,
one finds
\begin{equation} \label{e6} 
  I_1  :=  \int_{B(x_1,2^{(k+1)a_1})} |\, \phi_k (x-y)\cfi b (x-y)
  u (y)|^{t_1} \, dy_1
    \le  c_2  b_{k}^{t_1} 2^{ka_1}  \, M_1 |u|^{t_1}(x_1).  
\end{equation}
Proceeding iteratively by setting
$I_j = \int_{-\infty}^\infty (I_{j-1})^{t_j/t_{j-1}}\, dy_j$, 
one finds analogously
\begin{equation}
  \begin{split}
    I_n  & =  \int_{B(x_n,2^{(k+1)a_n})} (I_{n-1})^{t_n/t_{n-1}} \, dy_n
\\
     &\le  c_{n+1} \, b_{k}^{t_n}  
       2^{k t_n (\frac{a_1}{t_1}+\dots +\frac{a_{n-1}}{t_{n-1}})} 2^{ka_n} 
 M_n \big(\ldots (
   M_2 (M_1 |u|^{t_1})^{t_2/t_1}) \ldots \big)^{t_n/t_{n-1}} 
   (x_1, \ldots ,x_n).    
  \end{split}
\end{equation}
Raising to the power $1/t_n$ creates the factor 
$2^{k\vec a\cdot\frac1{\vec t}}$, so the desired inequality
follows from \eqref{e8} by observing that
$\sum_{k \in \Z} 2^{kd( \vec{a} \cdot \frac{1}{\vec{t}})}  
  \Norm{\cfi [\phi_k \, \cf b]}{L_1}^d 
  = \norm{b}{\dot{B}^{ \vec{a} \cdot \frac{1}{\vec{t}}, \vec{a}}_{1,d}}^d$.
\end{proof}

Now we turn to a vector-valued version which will be of great service
for us.

\begin{prop}\label{help1}
Suppose
$0 < t_k < \min (1,p_1,\ldots , p_k,q)$ for $k=1, \ldots, n$.
Let $\phi \in C_0^\infty(\Rn)$ such that $\supp \phi \subset
B_{\vec{a}}(0,2)$, and set
$\phi_j = \phi (2^{-j\vec a} \cdot)$, $j \in \N$. Then there exists a 
constant $c$ such that
\be
  \Norm{\cfi [\phi_j \cf u_j]}{\lpvec (\ell_q)(\Rn)}\le c
   R^{\vec{a} \cdot \frac{1}{\vec{t}} - |\vec a|}
   \norm{u_j}{\lpvec (\ell_q)(\Rn)}
\ee
for all sequences $\{u_j\}_{j=1}^\infty$ in $\cs'(\Rn)$ fulfilling
$\supp \cf u_j \subset \{\xi \mid |\xi|_{\vec a} \le R2^j \}$ for
some $R\ge1$.
\end{prop} 

\begin{proof}
Applying Proposition~\ref{Marschall} with $A= 2^j$ to 
$ \cfi [\phi_j \cf u_j]$, this is estimated by the iterated maximal function 
times 
$c(R2^j)^{\vec{a}\cdot\frac{1}{\vec{t}}-|\vec a|} 
\Norm{\phi (2^{-j\vec a} \cdot)}{\dot{B}^{\vec{a}\cdot\frac{1}{\vec{t}}, 
{\vec{a}}}_{1,d}}$. 
So by \eqref{Bhomo-eq}, 
\be
  |\cfi [\phi_j\cf u_j](x)|
 \le   c R^{\vec{a}\frac{1}{\vec{t}}-|\vec a|} 
  \Norm{\phi}{\dot{B}^{\vec{a}\cdot \frac{1}{\vec{t}}, {\vec{a}}}_{1,d}} 
  \big(M_n ( \dots (M_1 |u|^{t_1})^{t_2/t_1} \ldots )^{t_n/t_{n-1}}\big)^{1/t_n}
(x).
\ee
The claim now follows by repeated use of \eqref{bagby},
as in the proof of Proposition~\ref{maximalf}.
\end{proof}

The above techniques also give a proof of the lift property for the $\fspq$
scale.

\begin{prop}\label{lift}
The map $\Lambda_r\colon\cs'\to\cs'$ given by
$\Lambda_r u=\cfi [(1+|\xi|_{\vec a}^2)^{r/2}\cf u ]$ 
is a linear homeomorphism 
$\fspq(\Rn)\to F^{s-r ,\vec{a}}_{\vec{p},q} (\Rn)$ for every $r\in\R$. 
\end{prop}

\begin{proof}
To show the boundedness of $\Lambda_r$, we 
let $1=\sum \Phi_j$ denote the Littlewood--Paley decomposition;
and take $\phi_j=\Phi_{j-1}+\Phi_j+\Phi_{j+1}$
such that $\phi_j \Phi_j = \Phi_j $ for all $j$.
Moreover,
$\phi_j = \phi (2^{-j\vec a}\cdot)$ for $j\ge 1$ for a suitable $\phi$.
Then $\norm{\Lambda_r u}{F^{s-r,\vec a}_{\vec p,q}}$ consists of terms like
\be
  2^{(s-r)j}\cfi [ \Phi_j (1+ |\xi|_{\vec a}^2)^{r/2} \cf u]
 =2^{sj} \cfi [ g_j \Phi_j\cf u],
\ee
with Fourier multipliers
$g_j(\xi):= 2^{-rj} (1+ |\xi|_{\vec a}^2)^{r/2} \phi_j (\xi)$.
They fulfil $\supp g_j\subset\supp\phi_j\subset B_{\vec a}(0,R2^{j})$ for a
fixed $R\ge1$. Hence Marschall's inequality in Proposition~\ref{Marschall}
gives a bound of $|2^{sj}g_j(D)u_j(x)|$ by 
the iterated maximal function on $2^{sj}\Phi_j(D)u$ times
\be
\begin{split}
  c2^{j(\vec{a}\cdot \frac{1}{\vec{t}} -|\vec a|)} 
  \norm{g_j}{\dot{B}^{\vec{a}\cdot\frac{1}{\vec{t}},
  \vec{a}}_{1,d}}
  & \le 
  \Norm{ 2^{-rj} (1+ |2^{ja}\xi|_{\vec a}^2)^{r/2} \phi(\xi) 
   }{\dot{B}^{\vec{a} \cdot \frac{1}{\vec{t}}, \vec{a}}_{1,d}}
\\
  &\le c
  \Norm{( 2^{- 2j} + |\xi|_{\vec a}^2)^{r/2} \phi(\xi)}{W^{m}_1} =C
\end{split}
\ee
Here we have used the scaling property, and taken
some $m>\vec a\cdot\frac1{\vec t}$ to get a uniform bound for all $j\ge0$, 
which holds since $\phi=0$ around the origin (the case $j=0$ is obvious). 
Now boundedness of $\Lambda_r$ follows from Bagby's inequality, similarly to
the proof of Proposition~\ref{maximalf}. 
The estimates are valid for arbitrary $r\in \R$,
so the boundedness of $\Lambda_r^{-1}=\Lambda_{-r}$ is also obtained.
\end{proof}

\begin{rem}
The lift property in Proposition~\ref{lift} applies to the proof of 
Proposition~\ref{Littlewood}.
Indeed, for $H^{s,\vec a}_{\vec p}$ it will be enough to prove 
$H^{0,\vec{a}}_{\vec{p}} (\Rn) = F^{0,\vec{a}}_{\vec{p},2} (\Rn)$
with equivalent norms; but this holds by \eqref{lwp-ineq}.
(Kr\'ee's result \cite{Kr67} was also used in
\cite[Thm.~2]{Liz70} for the proof of a variant of \eqref{lwp-ineq} with a
homogeneous, but non-smooth decomposition.) 
For $m_k = s/a_k$, $k=1, \ldots, n$,
the identification $\wmp (\Rn) = \hsp (\Rn)$, with equivalent norms, has
been proved by Lizorkin, cf.\  Theorem~3 and \mbox{(20)} ff.\ in 
\cite{Liz70}.
\end{rem}


\subsection{Convergence criteria}


  \label{cc-ssect}
It is a central theme to conclude the
convergence in $\cs'$ of a series $\sum_{j=0}^\infty u_j$, 
where $\supp \cf u_j$ is compact for each $j$.
More precisely the $u_j$ are  supposed to satisfy one of the
following conditions, that can be imposed for each choice
of $\vec a$:
\begin{Rmlist}
  \item (The dyadic corona condition.)
    \label{DCC-cnd}
   There exist an $A>1$ such that for every $j\ge1$,
   \begin{equation}
   \supp \hat u_j\subset\{\,\xi\mid 
       \tfrac{2^j}{A}\le |\xi|_{\vec a}\le A2^j \,\},
   \end{equation}
   whilst $\supp \hat u_0\subset\{\,\xi\mid |\xi|_{\vec a}\le A\,\}$.
  \item (The dyadic ball condition.)
    \label{DBC-cnd}
   There exist an $A>0$ such that for every $j\ge0$,
   \begin{equation}
   \supp \hat u_j\subset\{\,\xi\mid  |\xi|_{\vec a}\le A2^j \,\}.  
   \end{equation}
\end{Rmlist}
The convergence of $\sum_{j=0}^\infty u_j$ will follow, if in addition to
one of these conditions either some growth or integrability condition is
fulfilled by the $u_j$ in a uniform way. The resulting dyadic corona and
dyadic ball \emph{criteria} are summed up below.

To conclude the mere $\cs'$-convergence, the following lemma was given 
for $\vec a=(1,\dots ,1)$
by Coifman and Meyer albeit without arguments \cite[Ch.~16]{MeCo91}.
We give a proof here, because some of the observations therein have
additional consequences, that are useful for the present paper.

\begin{lem}
  \label{CF-lem}
$1^\circ$~Let $(u_j)_{j\in \N_0}$ be a sequence of $C^\infty$-functions in
$\cs'(\Rn)$ that for suitable constants $C\ge0$, $m\ge0$ fulfils 
both \eqref{DCC-cnd} and
\begin{equation}
  |u_j(x)|\le C 2^{jm}(1+|x|)^m \text{ for all $j\ge0$}.
  \label{CF-cnd}  
\end{equation}
Then $\sum_{j=0}^\infty u_j$ converges in $\cs'(\Rn)$ to a distribution $u$,
for which $\hat u$ is of order $m$.

$2^\circ$~For every $u\in \cs'(\Rn)$ the conditions \eqref{DCC-cnd} and
\eqref{CF-cnd} are fulfilled by the $u_j$ defined from a quasi-homogeneous
Littlewood--Paley decomposition of $u$.
\end{lem}

Since any $\hat u\in \cs'$ is of finite order, the $\hat u_j$
in $2^\circ$ are at most of the same order.
Then there is some $m\ge0$ such that 
$|u_j(x)|\le c_j(1+|x|)^m$, 
by the Paley--Wiener--Schwartz Theorem,
which almost gives \eqref{CF-cnd}; but
the $j$-dependence is by $2^\circ$ not worse than $c_j={\cal O}(2^{mj})$.

\begin{proof}
In $2^\circ$ it is clear that 
$u_j(x)=c\dual{\hat u}{\Phi_j
e^{\im x\cdot\xi}}$ fulfils \eqref{DCC-cnd} and 
\begin{equation}
  |u_j(x)|\le c \sup\bigl\{\,(1+|\xi|)^m
    |D^\alpha_{\xi}(\Phi(2^{-j}\xi)e^{\im x\cdot\xi})| \bigm| 
     \xi\in \Rn,\quad |\alpha|\le m\,\bigr\}.
\end{equation}
Invoking Leibniz' rule, the worst terms occurs when derivatives of order $m$
fall on the exponential, and this is estimated by $C2^{jm}(1+|x|)^m$.
 
To prove $1^\circ$, note that 
if $\psi\in C^\infty(\Rn)$ is supported for
$\tfrac{1}{2A}\le|\xi|_{\vec a}\le2A$
and equalling $1$ where $\tfrac{1}{A}\le|\xi|_{\vec a}\le A$, 
any $\varphi\in \cs$ fulfils 
\begin{equation}
  |\dual{u_j}{\overline{\varphi}}|\le 
  \nrm{(1+|x|^2)^{-\tfrac{m+n}{2}}u_j}{2}
  \nrm{(1+|x|^2)^{\tfrac{m+n}{2}}\cfi (\psi(2^{-j\vec a}\cdot)\hat\varphi)}{2}.
\end{equation}
Here the first norm is ${\cal O}(2^{mj})$ by \eqref{CF-cnd}. For any $k>0$
Parseval--Plancherel's identity gives
\begin{multline}
   \nrm{(1+|x|^2)^{m+n}\cfi (\psi(2^{-j\vec a}\cdot)\hat\varphi)}{2}
\\
  \le\sum_{|\alpha+\beta|\le2m+2n}
   c_{\alpha,\beta} 2^{-j\alpha\cdot \vec a}\nrm{D^\alpha\psi}{\infty}
   \nrm{(1+|\xi|)^{k+n/2}D^{\beta}\hat\varphi}{\infty}
   \big(\int_{2^{j-1}/A}^\infty r^{-1-2k}\,dr\big)^{1/2}
\\
  \le c(A,k,m,n,\varphi,\psi)2^{-jk}.
\end{multline}
That is, $\dual{u_j}{\overline{\varphi}}={\cal O}(2^{(m-k)j})$ for $k>m$, so
$\sum_{j=0}^\infty \dual{u_j}{\varphi}$ converges, whence $\sum u_j$ does so
in $\cs'$.
\end{proof}

\begin{rem}
Littlewood--Paley decompositions 
$u=\sum_{j=0}^\infty u_j$ are \emph{rapidly convergent}, 
in the following sense:
if an arbitrary $u\in \cs'$ is decomposed as in 
$2^\circ$ above, the proof of $1^\circ$ gives 
\begin{equation}
  \dual{u_j}{\varphi}={\cal O}(2^{-Nj}) \quad\text{for every}\quad  N>0,
  \ \varphi\in \cs(\Rn),
  \label{rapid-eq}
\end{equation}
so
$\dual{u-\sum_{j< k}u_j}{\varphi}
=\sum_{j\ge k}\dual{u_j}{\varphi}={\cal O}(2^{-Nk})\to 0$, 
rapidly for $k\to\infty$.
\end{rem}

\bigskip

For the $\fspq$ we have the following 
(quasi-homogeneous) dyadic ball criterion:

\begin{lem}
  \label{Fball-lem}
When $s>\sum_{k=1}^n \tfrac{a_k}{\min(1,p_1,\dots,p_k,q)}-|\vec a|$ 
for $0<\vec p<\vec\infty$ and $0<q\le\infty$,
then there exists a $c>0$ such that, for every sequence
$(u_j)$ in $\cs'(\Rn)$ fulfilling
both the dyadic ball condition \eqref{DBC-cnd} and that
\begin{equation}
  F:=\Norm{(\sum_{j=0}^\infty 2^{sjq}|u_j|^q)^{\frac1q}}{L_{\vec p}}<\infty,
\end{equation}
the series $\sum_{j=0}^\infty u_j$ converges in $\cs'(\Rn)$ to a
$u\in\fspq(\Rn)$ for which
$\norm{u}{\fspq}\le cF$.
\end{lem}

\begin{proof} By condition \eqref{DBC-cnd} there is a fixed $h\in\N$ such that
$\Phi_j\cf u_k=0$ for $k<j-h$. So 
\be
  \cfi [\Phi_j\sum_{k=0}^M \cf u_k] = 
  \cfi [\Phi_j \sum_{k=j-h}^{M} \cf u_k ]
  \quad\text{for}\quad M\ge j-h.
\ee
Setting $k=j+\ell$
and using that $\norm{\cdot}{\ell_1}\le\norm{\cdot}{\ell_{\tau}}$
for $\tau=\min(1,p_1,\ldots, p_n,q)$, 
one obtains the first of the following inequalities, that also rely on
Proposition~\ref{help1} with $R=\max(1,A)2^{\ell_+}$,
\be \label{Fest-eq}
\begin{split}
 \Norm{\sum_{k\le M} u_k}{\fspq}^\tau
  &\le \Norm{\big(\sum_{j=0}^\infty (2^{sj\tau}\sum_{\ell=-h}^{M-j} 
             |\cfi[\Phi_j\cf u_{j+\ell}]|^\tau)^{q/\tau} \big)^{\tau/q}}
            {L_{\vec p/\tau}}
\\
 &\le  \sum_{\ell=-h}^M
 \Norm{2^{js} \cfi [ \Phi_j\cf u_{j+\ell}]}{\lpvec (\ell_q)}^\tau 
\\
  & \le   c
   \sum_{\ell=-h}^\infty 2^{\ell_+ \tau(\vec{a} \cdot \frac{1}{\vec{t}} -|\vec a|)}
  \Norm{2^{js}u_{j+\ell}}{\lpvec (\ell_q)}^\tau
  \le c_1 F^{\tau} \sum_{\ell=-h}^\infty 
  2^{\ell_+ \tau (-s +\vec{a} \cdot \frac{1}{\vec{t}} -|\vec a|)} .
\end{split}
\ee
Hereby $t_k < \min(1, p_1, \ldots,p_k , q)$ must be fulfilled.
But the $t_k$ can be taken with this property at the same time as 
$s > \vec{a} \cdot \frac{1}{\vec{t}} - |\vec a|$;
cf.\ the conditions on $s$ in the lemma.

With $\vec t$ as above, the sequence 
$(\sum_{k=0}^M u_k)_{M\in\N}$ is by \eqref{Fest-eq} bounded in
$\fspq (\Rn)$. 
Therefore 
it is fundamental in $F^{s',\vec a}_{\vec p,1}$ for $s'<s$, 
hence convergent to some $u$. 
Using Fatou's lemma for $M\to\infty$ on the left in \eqref{Fest-eq}, 
the estimate $\norm{u}{\fspq}\le cF$ is obtained.  
\end{proof}

In case $p_n \le\ldots\le p_2 \le p_1 $
the restriction for $s$ reduces to
$s>  \sum_{k=1}^n   a_k (
\frac{1}{\min(1,p_k, q)} - 1)$.
In case 
$p_1 =\ldots= p_n$ this gives back the unmixed version known since \cite{Y1}.

The above proof gives more, for if the series fulfils the 
stronger corona condition \eqref{DCC-cnd}, then $\cfi(\Phi_j\cf u_k)=0$ unless
$j-h\le k\le j+h$. In this case the sums in \eqref{Fest-eq} have 
$l\in\{-h,\dots,h\}$, so the restriction on $s$ is not needed. This proves

\begin{lem}
  \label{Fcor-lem}
When $s\in\R$ and
$0<\vec p< \vec\infty$, $0<q\le\infty$,
there exists $c>0$ such that, for every sequence
$(u_j)$ in $\cs'(\Rn)$ fulfilling
both the dyadic corona condition \eqref{DCC-cnd} and that
\begin{equation}
  F:=\Norm{(\sum_{j=0}^\infty 2^{sjq}|u_j|^q)^{\frac1q}}{L_{\vec p}}<\infty,
\end{equation}
the series $\sum_{j=0}^\infty u_j$ converges in $\cs'(\Rn)$ to a
$u\in\fspq(\Rn)$ for which
$\norm{u}{\fspq}\le cF$.
\end{lem}

For the Besov spaces, the dyadic ball and corona criteria follow by
interchanging the order of the $L_{\vec p}$ and $\ell_q$-norms 
in the proof Lemma~\ref{Fball-lem}, and by using Proposition~\ref{help1} 
for sequences having only a single non-trivial term. 
Thus one has the next result.

\begin{lem}
  \label{B-lem}
When $s>\sum_{k=1}^n \tfrac{a_k}{\min(1,p_1,\dots,p_k)}-|\vec a|$ 
for $0<\vec p\le\vec\infty$ and $0<q\le\infty$,
there exists $c>0$ such that, for every sequence
$(u_j)$ in $\cs'(\Rn)$ fulfilling
both \eqref{DBC-cnd} and 
\begin{equation}
  B:=(\sum_{j=0}^\infty 2^{sjq}\norm{u_j}{L_{\vec p}}^q)^{\frac1q}<\infty,
\end{equation}
the series $\sum_{j=0}^\infty u_j$ converges in $\cs'(\Rn)$ to a
$u\in\bspq(\Rn)$ for which
$\norm{u}{\bspq}\le cB$.

If $B<\infty$ and \eqref{DCC-cnd} hold, then
the convergence and $\norm{u}{\bspq}\le cB$ holds for all
$s\in\R$.
\end{lem}

By Lemma~\ref{Fcor-lem} and \ref{B-lem}, 
the choice of the Littlewood--Paley decomposition
and the constants are without 
significance for the $\fspq$ and $\bspq$ spaces.  
For completeness the next result is given.

\begin{lem}
Every differential operator of the form
$D^\alpha=D^{\alpha_1}_{x_1}\dots D^{\alpha_n}_{x_n}$ gives 
continuous maps
$\fspq(\Rn)\to F^{s-\alpha\cdot\vec a,\vec a}_{\vec p,q}(\Rn)$ and
$\bspq(\Rn)\to B^{s-\alpha\cdot\vec a,\vec a}_{\vec p,q}(\Rn)$, for every $s\in\R$.
\end{lem}
\begin{proof}
For the scale $\fspq$,
Lemma~\ref{Fcor-lem} and Proposition~\ref{help1} applied to
the decomposition
$D^\alpha u=\sum_{j=0}^\infty (D^\alpha \cfi\Phi_j)*u$ give at once that
$D^\alpha$ has order $\alpha\cdot\vec a$.
The Besov case is similar.
\end{proof}

As another consequence of the dyadic corona criterion, we sketch a

\begin{proof}[{P\,r\,o\,o\,f} of Lemma~\ref{SFS-lem}]
The embeddings \eqref{SFS'-eq}--\eqref{SBS'-eq} were shown in
\cite[Prop.~10]{JoSi}. 
The density of $\cs\subset\fspq$ follows from
Lemma~\ref{Fcor-lem}: $u^N:=\sum_{j=0}^N \Phi_j(D)u$ converges to $u$ in
$\fspq$, because for $u-u^N=\sum_{j>N}\Phi_j(D)u$ the number $F\to0$ as
$N\to\infty$ by dominated convergence ($q<\infty$). The set of $g\in
L_{\vec p}\cap\cs'$ with $\supp \cf g\subset B_{\vec a}(0,2^{N+1})$ is 
embedded into $\fspq$, for $g=g+0+\dots$ fulfils \eqref{DCC-cnd}
with $A=2^{N+1}$. Therefore the convergence of 
$u^N\cdot c\cfi\Psi_0(\varepsilon\cdot)\in \cs$ to $u^N$ in $L_{\vec p}$ for
$\varepsilon\to 0$ implies
$\norm{c\cfi\Psi_0(\varepsilon\cdot)u^N-u^N}{\fspq}\to0$. A similar reasoning
works for $\bspq$.
\end{proof}

\bigskip

Occasionally it is useful to have a corona criterion based on powers of
$2^\lambda$ for some $\lambda>0$.

\begin{lem}
  \label{Fcor'-lem}
When $s\in\R$ and
$0<\vec p< \vec\infty$, $0<q\le\infty$,
there exists $c>0$ such that, for every sequence
$(u_j)$ in $\cs'(\Rn)$ fulfilling $\supp \cf u_0\subset B_{\vec a}(0,A)$ and
\begin{gather}
    \supp \cf u_j\subset \{\,\xi\mid  \tfrac{1}{A}2^{\lambda j}
   \le |\xi|_{\vec a}\le A 2^{\lambda j} \,\}
 \quad\text{for}\quad j\ge1,
  \label{Ilambda-eq}
    \\
  F_{\lambda}:=
  \Norm{(\sum_{j=0}^\infty |2^{\lambda sj}u_j|^q)^{\frac1q}}
       {L_{\vec p}}<\infty,
\end{gather}
the series $\sum_{j=0}^\infty u_j$ converges in $\cs'(\Rn)$ to a
$u\in\fspq(\Rn)$ for which
$\norm{u}{\fspq}\le cF_\lambda$.
\end{lem}
\begin{proof}
Note that \eqref{Ilambda-eq} 
gives an $h\in \N$ such 
$\Phi_j\cf u_k=0$ unless $\tfrac{j}{\lambda}-h\le k\le
\tfrac{j}{\lambda}+h$. With $k=[j/\lambda]+\nu$ 
($[\cdot]$ is the integer part), 
a modification of \eqref{Fest-eq} gives
\be \label{Fest'-eq}
\begin{split}
 \Norm{\sum_{k\le M} u_k}{\fspq}^\tau
  &\le \Norm{\big(\sum_{j=0}^\infty (2^{sj\tau}\sum_{\nu =-h}^{h+1} 
         |\cfi[\Phi_j\cf u_{[j/\lambda]+\nu}]|^\tau)^{q/\tau} \big)^{\tau/q}}
         {L_{\vec p/\tau}}
\\
  & \le   c
   \sum_{|\nu|\le h+1} 
  (A2^{|\nu|\lambda})^{\tau(\vec{a} \cdot \frac{1}{\vec{t}} -|\vec a|)}
  \Norm{2^{js}u_{[j/\lambda]+\nu}}{\lpvec (\ell_q)}^\tau.
\end{split}
\ee
Here the last inequality results from Proposition~\ref{help1}, for
$\xi\in\supp u_{[j/\lambda]+\nu}$ entails 
$|\xi|_{\vec a}\le A2^{\lambda([j/\lambda]+\nu)}\le (A2^{\lambda|\nu|})2^j$.
It is clear that $2^{sj}\le c2^{s\lambda[j/\lambda]}$. 
Therefore $m=[j/\lambda]$ gives
$\norm{2^{js}u_{[j/\lambda]+\nu}}{\ell_q}\le c_{\lambda}
\norm{2^{sm\lambda}u_{m}}{\ell_q}$, for the sequence 
$(2^{js}u_{[j/\lambda]+\nu})_{j\in\N_0}$ is either lacunary for
$0<\lambda<1$ or, for $\lambda\ge1$, it has every $u_{m+\nu}$ repeated 
at most $[\lambda]+1$ times. Consequently
$ \Norm{\sum_{k\le M} u_k}{\fspq}\le cF_\lambda$ for all $M$, so that
convergence and the estimate follow as in the proof of Lemma~\ref{Fcor-lem}.
\end{proof}
 
For example Lemma~\ref{Fcor'-lem} gives invariance of 
the spaces 
$\fspq$ under the 
reparametrisation $(s,\vec a)\mapsto (\lambda s,\lambda\vec a)$:

\begin{lem}   \label{Flambda-lem}
$\fspq(\Rn)=F^{\lambda s,\lambda\vec a}_{\vec p,q}(\Rn)$ 
for every $\lambda>0$, and the quasi-norms are equivalent.
\end{lem}
\begin{proof}
For 
$\vec b=\lambda\vec a$ the definition gives 
$|\xi|_{\vec b}^\lambda=|\xi|_{\vec a}$, so that the Littlewood--Paley
decomposition $1=\sum _{j=0}^\infty \Phi^{\vec b}_j$
associated with $\vec b$
yields functions that for $j\ge1$ are equal to $1$ in the set where
$(\tfrac{13}{20})^\lambda 2^{\lambda j}\le|\xi|_{\vec a}\le
 (\tfrac{11}{10})^\lambda 2^{\lambda j}$. 
Hence Lemma~\ref{Fcor'-lem} gives $\norm{u}{\fspq}\le c
\norm{u}{F^{\lambda s,\lambda\vec a}_{\vec p,q}}$. Since $\vec a$ and
$\lambda>0$ are arbitrary, the opposite inequality also holds.
\end{proof}

\begin{rem}   \label{veca-rem}
In view of this lemma, we may assume that all $a_k\ge1$, which is convenient 
in Section~\ref{prfs-sect} below. 
However, this is immaterial for the statements in Section~\ref{main-sect},
since the inequalities \eqref{spq1-eq}, \eqref{spq-eq} etc.\  
hold for some $s$, $\vec a$ if and only if
they hold for all $\lambda s$, $\lambda\vec a$, $\lambda>0$.
Hence $\vec a\in \,]0,\infty[\,^n$ is assumed in Section~\ref{main-sect}.
\end{rem}

\begin{rem}
Since there are few general references to the mixed
norm spaces $\fspq$, we note that the reader may find the
necessary theory here and in \cite{JoSi}.  
\end{rem}

\section{Proofs}
  \label{prfs-sect}

\subsection{The general necessary conditions} We first give the proof of 
Lemma~\ref{spq-lem}, since this just amounts to a calculation of some norms
in $\fspq$ of suitably chosen functions.
Recall that we can normalise to $\min(a_1,\dots,a_n)=1$, cf
Remark~\ref{veca-rem}. 

\subsubsection{Examples}

To have a convenient set-up, we shall consider traces on the 
hyperplane $x_m=0$ for arbitrary $m\in\{1,\dots,n\}$. The remaining $n-1$ variables are split in two groups $x_\ge$ and $x_<$.
The reason for this labelling will be clear later when a $\vec p$ is fixed: 
the components $p_k$ with $k\ne m$ splits naturally into the groups
$p_\ge$ and $p_<$ in which $p_k\ge1$, respectively $p_k<1$; 
accordingly $x_\ge$, $x_<$ are defined from the same indices.

Let $f$, $g\in \cs(\R)$ be fixed, as we may, such that 
$\int_{\R}f(t)\,dt =1$, $g(0)=1$ and, with $a_0=\max(a_1,\dots,a_n)$,
\begin{equation} 
  \supp\hat f\subset\{\,|\tau|<1/(10n)^{a_0} \,\},\qquad
  \supp\hat g\subset\{\,(\tfrac8{10})^{a_m}\le|\tau|\le1 \,\}.
\end{equation}
Introducing the tensor product
\begin{equation}
  w_l(x)= (\prod_{x_{\ge}}f(x_k))\otimes  g(2^{la_m}x_m)
    \otimes (\prod_{x_{<}} 2^{la_k}f(2^{la_k}x_k))
\end{equation}
we shall estimate the Schwartz function $v_j=\tfrac1j\sum_{l=j+1}^{2j} w_l$.
Note first that for $\xi\in\supp\hat w_l$, one has for 
the vector $\eta=\xi-\xi_m e_m$
(formed by resetting the $m^{\op{th}}$ coordinate to $0$)
that, since $\tfrac{a_0}{a_k}\ge1$ for all $k$,
\begin{equation}
  |\eta|_{\vec a} \le \sum_{k\ne m} |\xi_k|^{1/a_k}
  \le \sum_{x_\ge} (10n)^{-\frac{a_0}{a_k}}
     + \sum_{x_<} 2^l(10n)^{-\frac{a_0}{a_k}}\le \tfrac{n-1}{10n}\cdot 2^l.
\end{equation}
Using the triangle inequality for $|\cdot|_{\vec a}$,
\be
  \tfrac{7}{10}2^l\le |\xi_m|^{1/a_m}-|\eta|_{\vec a}\le
  |\xi|_{\vec a}\le|\xi_m|^{1/a_m}+|\eta|_{\vec a}< \tfrac{11}{10}2^l.
\ee
This means that every $\xi\in\supp\hat w_l$ satisfies $\Phi_l(\xi)=1$, 
for this identity holds where
$\tfrac{13}{20}2^l\le|\xi|_{\vec a}\le \tfrac{11}{10}2^l$. 
Consequently the $\Phi_l$ disappear from the norms of $v_j$, e.g.\ 
\begin{equation}
  \norm{v_j}{\fspq}=\tfrac1j \Norm{(\sum_{l=j+1}^{2j} 
    2^{slq}|w_l(\cdot)|^q)^{1/q}}{L_{\vec p}}.
  \label{normvj-eq}
\end{equation}
For certain triples $(s,\vec p,q)$ this can be calculated precisely.

\begin{lem} \label{normvj-lem}
  Let $\vec p$ be a vector in $]0,\infty]^n$, and let $p_\ge$ and $p_<$ be the 
above mentioned splitting corresponding to a fixed $m$.

$1^\circ$~For $s=\tfrac{a_m}{p_m}+\sum_{k\ne m}(\tfrac{a_k}{p_k}-a_k)_+$
it holds for every $q$ that
\begin{equation}
  \norm{v_j}{\bspq}=c\cdot j^{\frac1q-1}.
\end{equation}

$2^\circ$~If $p_m>1$ and $p_k\ge1$ for $k\ne m$, then for
$s=\tfrac{a_m}{p_m}$, 
\begin{equation}
  \norm{v_j}{F^{s,\vec a}_{\vec p,p_m}}
  =c\cdot j^{\frac1{p_m}-1}.
\end{equation}
\end{lem}

\begin{proof}
In analogy with \eqref{normvj-eq} above, 
$  \norm{v_j}{\bspq}=\tfrac1j (\sum_{l=j+1}^{2j} 
    2^{slq}\norm{w_l}{L_{\vec p}}^q)^{1/q}$.
Since the $L_{\vec r}$-norm respects the tensor products entering $w_l$, 
and since $2^{l(\frac {a_m}{p_m}+\sum_{p_<}(\frac{a_k}{p_k}-a_k))}$
is absorbed by the dilations, 
$\norm{v_j}{\bspq}
  =
  \tfrac 1j \prod_{k\ne m}\nrm{f}{p_k}
  (\sum_{l=j+1}^{2j} \nrm{g}{p_m}^q  )^{\tfrac{1}{q}}
  =
  c j^{\frac1q-1}
$.

In case $2^\circ$, a similar procedure applies to \eqref{normvj-eq};
the group $x_<$ is empty by assumption, so 
\begin{equation}
  \norm{v_j}{F^{s,\vec a}_{\vec p,p_m}} 
  = \tfrac{1}{j} \prod_{k\ne m}\nrm{f}{p_k}
  (\sum_{l=j+1}^{2j}   
   \int_\R   2^{l a_m} |g(2^{la_m}x_m)|^{p_m}\,dx_m)^{\frac1{p_m}}
  = c\cdot j^{\frac1{p_m}-1}
\end{equation}
since the factors involving 
$f$ do not depend on the summation index.
\end{proof}

The interest of Lemma~\ref{normvj-lem}
comes from the obvious fact that
\begin{equation}
  \gamma_{0,m} v_j\to \delta_0(x_{<})
  \otimes \prod_{x_{\ge}} f(x_k)\quad\text{for}\quad j\to\infty
  \label{g0mvj-eq}
\end{equation}
(which means $f(x_1)\otimes\dots\otimes f(x_n)$ if $x_{<}$ is empty).
From this we get the
 
\subsubsection{Proof of Lemma~\ref{spq-lem}}
Given that $\gamma_{0,m}\colon \fspq(\Rn)\to\cd'(\Rnl)$
is continuous for some $(s,\vec p,q)$,
we set $t=\frac{a_m}{p_m}+\sum_{p_k\ne m}(\frac{a_k}{p_k}-a_k)_+$.

Then $s<t$ cannot hold, 
for else $B^{t,\vec a}_{\vec p,2} \imb \fspq$, and
this embedding would be incompatible with the
continuity of $\gamma_{0,m}$, since by Lemma~\ref{normvj-lem}
the $v_j$ tend to $0$ in $B^{t,\vec a}_{\vec p,2}$ and a fortiori in
$\fspq$
(whilst $\gamma_{0,m}v_j\not\to 0$, cf.~\eqref{g0mvj-eq}). 
Therefore the continuity implies 
$s\ge\frac {a_m}{p_m}+\sum_{p_k\ne m}(\frac{a_k}{p_k}-a_k)_+$.
  
Similarly $2^\circ$ of Lemma~\ref{normvj-lem} shows that
in case $p_k\ge1$ for $k\ne m$, 
the trace $\gamma_{0,m}$ is only continuous from $\fspq$ on the borderline 
(which is $s=a_m/p_m$ then) if $p_m\le1$.

\subsection{Proof of Theorem~\ref{K-thm}}
We shall proceed with Theorem~\ref{K-thm}, for later we
draw on the properties of the extension operator, during the proof of
the theorems on the trace.

The next well-known lemma plays a significant role in the proofs, 
e.g.\ because the property of $K_1$ and $K_n$ that they
map into $\bigcap_{0<q\le\infty}\fspq$ is a consequence of the fact that 
both \eqref{lspos-ineq} and \eqref{lsneg-ineq} hold for \emph{any}
$\ell_r$-norm, $0<r\le\infty$.

\begin{lem}
  \label{Y-lem}
If $(b_j)_{j\in N_0}$ is a sequence of complex numbers,
$s>0$ and $q$, $r\in \,]0,\infty]$, there is a constant 
$c=c(s,q,r)$ such that (with sup-norm over $k$ for $r=\infty$)
\begin{align}
  \Norm{\{2^{sj}(\sum_{k=j}^{\infty}|b_k|^r)^{1/r}\}_{j=0}^\infty}{\ell_q}
  &\le c\Norm{\{2^{sj}b_j\}_{j=0}^\infty}{\ell_q}
  \label{lspos-ineq}  \\
  \Norm{\{2^{-sj}(\sum_{k=0}^{j}|b_k|^r)^{1/r}\}_{j=0}^\infty}{\ell_q}
  &\le c\Norm{\{2^{-sj}b_j\}_{j=0}^\infty}{\ell_q}.
  \label{lsneg-ineq}
\end{align}
\end{lem}
For $r=1$ this lemma is equivalent to \cite[Lem.~3.8]{Y1}; in general it may
be proved in a similar fashion as noted in \cite[Lem.~2.5]{JJ96ell}.

\subsubsection{The right-inverse $K_1$} 
Note first that $\varphi_j(\xi''):=\Phi_j(0,\xi'')$
gives a Littlewood--Paley decomposition on $\Rnl$, 
so any $v\in \cs'(\Rnl)$ may be written $v=\sum v_j$ for
$v_j=\varphi_j(D)v$.

To construct $K_1$ we introduce an auxiliary function
$\cf\psi\in C_0^\infty(\R)$ such that $\psi(0)=1$ and
$\supp\cf\psi\subset[1,2]$.
Then $K_1$ can be defined as
\begin{equation}
  K_1v(x) = \sum_{j=0}^\infty \psi (2^{ja_1}x_1)v_j(x''),
  \label{K1-eq}
\end{equation}
for the series converges in $\cs'$ by Lemma~\ref{CF-lem}.
To verify this, note that
$\cf(\psi(2^{ja_1}\cdot)v_j)$ equals the product
$2^{-ja_1}\hat\psi(2^{-ja_1}\xi_1)\varphi_j(\xi'')\hat v(\xi'')$, 
where e.g.\ $1\le|2^{-a_1j}\xi_1|\le 2$ implies
$2^{a_1j}\le|\xi_1|\le2^{a_1(j+1)}$
and
\begin{equation}
  |\xi_1|^{1/a_1}\le |(\xi_1,\xi'')|_{\vec a}
  \le |(\xi_1,0)|_{\vec a}+|(0,\xi'')|_{\vec a}
 \le |\xi_1|^{1/a_1} +|\xi''|_{a''};
\end{equation}
this immediately give the inclusions, valid for $j\ge0$, 
\begin{equation}
  \supp \cf(\psi(2^{ja_1}\cdot)v_j)
  \subset \{\,\xi\mid 
            2^j\le |\xi|_{\vec a}\le4\cdot 2^j \,\}.
\end{equation}
Moreover, from $2^\circ$ in Lemma~\ref{CF-lem} the growth condition
\eqref{CF-cnd} follows at once. 
Hence $K_1$ is a well defined linear map $\cs'(\Rnl)\to\cs'(\Rn)$. 

Furthermore, 
$\Lambda\colon x_1\to \sum_{j=0}^\infty \psi(2^{ja_1}x_1)v_j(x'')$
is in the set $C_{\op{b}}(\R,\cs'(\Rnl))$ of continuous bounded maps
$\R\to\cs'(\Rnl)$.
In fact, the functions $\psi(2^{ja_1}\cdot)$ are uniformly bounded,
so that $g(x_1)=\sum \psi(2^{ja_1}x_1)\dual{v_j}{\varphi}$ by
\eqref{rapid-eq} converges to a continuous and bounded function on $\R$.
Hence $x_1\mapsto\dual{\Lambda(x_1)}{\varphi}$ has these properties, so
$\Lambda\in C_{\op{b}}(\R,\cs'(\Rnl))$.

For every $\eta\in \cs(\Rn)$ this implies the first identity in
\begin{equation}
  \begin{split}
  \dual{\Lambda}{\eta}&= 
  \int_{\R} \dual{\Lambda(x_1)}{\eta(x_1,\cdot)}_{\Rnl}\,dx_1
  = \int\sum_{j=0}^\infty\dual{\psi(2^{ja_1}x_1)v_j}{\eta(x_1,\cdot)}dx_1
\\
  &= \lim_{m\to\infty}\sum_{j=0}^m 
    \dual{\psi(2^{ja_1}\cdot)v_j}{\eta}_{\Rn} 
   =\dual{K_1v}{\eta}.
  \end{split}
  \label{K1L-eq}
\end{equation}
Here passage to the last line is justified with the following majorisation,
\begin{equation}
  \sup_{x_1}|\dual{v_j}{\eta(x_1,\cdot)}|\le C_N2^{-jN}(1+x_1^2)^{-1} ,
  \quad\text{for every $N>0$}
  \label{K1L-ineq}
\end{equation}
that follows analogously to \eqref{rapid-eq}, by taking for $\varphi$ in the
proof of \eqref{rapid-eq} a function like
$\varphi_t=(1+t^2)\eta(t,x)$ depending on a parameter $t$.

By the above formula 
$K_1v=\Lambda\in C(\R,\cs'(\Rnl))$, so since $\psi(0)=1$,
\begin{equation}
  \gamma_{0,1} K_1v=\Lambda(0)= \sum_{j=0}^\infty \psi(0)v_j =v
  \quad\text{for every $v\in \cs'(\Rnl)$}.
\end{equation}
That is, $K_1$ maps all of $\cs'(\Rnl)$ into the domain of $\gamma_{0,1}$,
for which it acts as a right-inverse.

Continuity of $K_1\colon \cs'(\Rnl)\to\cs'(\Rn)$ 
results by proving that there exists an everywhere defined linear map 
$K_1^*\colon \cs(\Rn)\to\cs(\Rnl)$ given by 
\begin{equation}
  K_1^*\eta(x'')=\sum_{j=0}^\infty \int_{\R} \psi(2^{ja_1}y_1)
  \int_{\Rnl}\cf^{-1}\varphi_{j}(y'')\eta(y_1,x''-y'')\,dy''dy_1.
  \label{K1*-eq}
\end{equation}
Indeed, using $K_1^*$ one arrives at the following formula, where the right
hand side depends continuously on $v\in\cs'(\Rnl)$,
\begin{equation}
  \dual{K_1v}{\overline{\eta }}=
  \dual{v}{\sum_{j=0}^\infty \big(\int_{\R}
    \psi(2^{ja_1}y_1)\cf^{-1}_{\xi''\to x''}
              (\varphi_j\cf_{x''\to\xi''}\eta)
  \,dy_1\big)^{\overline{\ }}}_{\Rnl}=
  \dual{v}{\overline{K_1^*v}}.
\end{equation}
As for \eqref{K1*-eq} it is noted that $\cs(\Rn)$ contains
\begin{equation}
  (\hat \psi(-\xi_1)\Phi_0(0,\xi'')+
  \sum_{j=1}^\infty  2^{-ja_1}\hat \psi(-2^{-ja_1}\xi_1)
  \Phi_1(0,2^{-(j-1)a''}\xi''))\cf\eta(\xi_1,\xi''),
\end{equation}
since this is a product of $\cf\eta\in \cs$ and a $C^\infty$-function
with bounded derivatives. Applying $\cf^{-1}$ and setting $x_1=0$,
it results that the right-hand side of \eqref{K1*-eq} is in $\cs(\Rnl)$ .

\subsubsection{Boundedness of $K_1$}
With $v\in F^{s-\frac{a_1}{p_1},a''}_{p'',p_1}(\Rnl)$, for $s\in\R$,
we obtain boundedness of $K_1$ by showing that the series defining 
$K_1v$ converges in $\fspq(\Rn)$.
For this it suffices  by Lemma~\ref{Fcor-lem} to show
\begin{equation}
  \Norm{\sum_{j=0}^\infty \psi(2^{ja_1}x_1)v_j(x'')}{L_{\vec p}(\ell^s_q)}
  \le c\norm{v}{F^{s-\frac{a_1}{p_1},a''}_{p'',p_1}}.
  \label{K1FF-ineq}
\end{equation}
By embeddings this may be reduced to the case $q<p_1$. For the integral
\begin{equation}
  I(x''):=\int_{\R}
   (\sum_{j=0}^\infty |2^{sj}\psi(2^{ja_1}x_1)
                             v_j(x'')|^q)^{\tfrac{p_1}{q}}\,dx_1
\end{equation}
we take $N>\frac1{p_1}$ so that
$  |\psi(2^{ja_1}x_1)|\le
|2^{ja_1}x_1|^{-N}\sup_{\R}t^N|\psi(t)|$ for $x_1\ne 0$.
Then, if $I_1$ and $I_0$ denote the integrals over $|x_1|>1$ and
$|x_1|\le1$, respectively, 
\begin{equation}
  \begin{split}
  I_1&\le \int_{|x_1|>1} 
       (\sum_{j=0}^\infty |2^{sj}v_j(x'')|^q2^{-Na_1jq}c_\psi)^{\tfrac{p_1}{q}}
       x_1^{-Np_1}\,dx_1 
\\
  &\le c_1(1-2^{(\frac1{p_1}-N)a_1q})^{-\tfrac{p_1}{q}}
    (\sup_{j}2^{(s-\frac{a_1}{p_1})j}|v_j(x'')|)^{p_1}.
  \end{split}
\end{equation}
By splitting the integration area for $I_0$ into intervals 
with $2^{-(k+1)a_1}\le|x_1|\le 2^{-ka_1}$,
that are of length $(2-2^{1-a_1})2^{-ka_1}$,
and by using the choice of $N$ for $j>k$,
\be
  I_0 \le \sum_{k=0}^\infty c2^{-ka_1}(
        \sum_{j=0}^k |2^{sj}v_j(x'')|^q\nrm{\psi}{\infty}^q
        +\sum_{j=k+1}^\infty
        |v_j(x'')2^{(s-Na_1)j+N(k+1)a_1}c(\psi)|^q
        )^\frac{p_1}q .
\ee
At the cost of a factor of $2^{\frac{p_1}q}$ the two terms may be treated
separately, so
\be 
  I_0 \le c_2\sum_{k=0}^\infty 2^{-ka_1}(
        \sum_{j=0}^k |2^{sj}v_j(x'')|^q)^{\frac{p_1}q}
        +c_3 \sum_{k=0}^\infty 2^{k(Na_1-\frac{a_1}{p_1})p_1}
        (\sum_{j=k}^\infty
        |v_j(x'')2^{(s-Na_1)j}|^q
        )^\frac{p_1}q .
\ee 
According to Lemma~\ref{Y-lem}, the $\ell_q$-norms over $j$ may be
``cancelled'' since the weights have 
bases $2^{-a_1}<1$ and $2^{(N-\frac1{p_1})a_1p_1}>1$, respectively, so
\begin{equation}
  I_0\le (c_2+c_3)\norm{2^{(s-\frac{a_1}{p_1})j}v_j(x'')}{\ell_{p_1}}^{p_1}.
\end{equation}
Altogether 
$I(x'')\le c_4 \norm{2^{(s-\frac{a_1}{p_1})j}v_j(x'')}{\ell_{p_1}}^{p_1}$,
so by continued calculation of the $L_{\vec p}\,$-norm, 
\eqref{K1FF-ineq} follows. 
Therefore $K_1$ is bounded 
$ F^{s-\frac{a_1}{p_1},a''}_{p'',p_1}(\Rnl)\to \fspq(\Rn)$ 
for all $s\in \R$, $q>0$.

\subsubsection{The extension operator $K_n$} This is in analogy with 
$K_1$ taken as
\begin{equation}
    K_nv(x) = \sum_{j=0}^\infty \psi (2^{ja_n}x_n)v_j(x').
  \label{Kn-eq}
\end{equation}
By Lemma~\ref{CF-lem}, this is also meaningful in $\cs'$, and the above
discussion, mutatis mutandis, gives that 
$K_n$ is a right-inverse of $\gamma_{0,n}$.

To show that $K_n$ is bounded from $B^{s-\frac{a_n}{p_n},a'}_{p',p_n}(\Rnl)$
to $\fspq(\Rn)$ for all $q\in \,]0,\infty]$, we may assume that
$q<\min(p_1,\dots,p_n)$. For $v$ belonging to the former space, we set
\begin{equation}
    I:=\int_{\R}
    \Norm{
    (\smash{
     \sum_{j=0}^\infty|2^{sj}\psi(2^{ja_n}x_n)v_j(\cdot)|^q)^{\tfrac1{q}}
    }}{L_{p'}}^{p_n}
   \,dx_n.
\end{equation}
For the integral $I_1$ over
$|x_n|\ge1$, one can use an $N>\tfrac{1}{p_n}$ (but otherwise as above) 
together with the triangle inequality for the mixed-norm  with exponent
$\frac1q p'=(\tfrac{p_1}{q},\dots,\tfrac{p_{n-1}}{q})$ to obtain that
\begin{equation}
    I_1\le\int_{|x_n|\ge 1}
    \smash{
    (\sum_{j=0}^\infty
    \norm{2^{sjq}|v_j|^q }{L_{\tfrac{1}{q}p'}}
    c_\psi2^{-Na_njq}
    )^{\frac{p_n}{q}}
    }
    x_n^{-Np_n}\,dx_n.
\end{equation}
Since $q<p_n$, H{\"o}lder's inequality gives
$I_1\le c\norm{v}{B^{s-\frac{a_n}{p_n},a'}_{p',p_n}(\Rnl)}^{p_n}$.

Correspondingly $I_0$ is split into regions with $2^{-(k+1)a_n}\le
|x_n|\le 2^{-ka_n}$ and this yields, cf.\ the case for $K_1$ above,
\begin{multline}
  I_0 \le c_1\sum_{k=0}^\infty 2^{-ka_n}(
        \sum_{j=0}^k \norm{2^{sjq}|v_j|^q}{L_{\frac1q p'}})^{\frac{p_n}q}
\\
        +c_2 \sum_{k=0}^\infty 2^{k(Na_n-\frac{a_n}{p_n})p_n}
        (\sum_{j=k}^\infty
        \norm{|v_j|^q}{L_{\frac1q p'}}2^{(s-Na_n)jq}
        )^\frac{p_n}q .
\end{multline}
By passing to the $L_{p'}$-norms and applying Lemma~\ref{Y-lem}, one can
get rid of the sums over $j\lesseqgtr k$, hence
$I\le c \norm{v}{B^{s-\frac{a_n}{p_n},a'}_{p',p_n}}^{p_n}$. This shows 
that $K_n$ is continuous 
$B^{s-\frac{a_n}{p_n},a'}_{p',p_n}(\Rnl)\to \fspq(\Rn)$ 
for $0<q\le\infty$, any $s\in \R$.

\begin{rem}
  \label{Tr-rem}
Our treatment of $K_1$ and $K_n$ 
was inspired by the isotropic estimates in \cite[Thm.~2.7.2]{T2}.
We have preferred to use Lemma~\ref{Y-lem} and the dyadic corona criterion, 
that also give that the $K_m$ map all of 
$\cs'(\Rnl)$ into the domain of $\gamma_{0,m}$.
The continuity $K_m\colon \cs'(\Rnl)\to\cs'(\Rn)$ followed from the
existence of an adjoint $K_m^*\colon \cs(\Rn)\to\cs(\Rnl)$.
\end{rem}

\subsection{On Corollaries~\ref{gj1-cor}--\ref{gjn-cor}}
As noted prior to the corollaries, boundedness follows directly from the
other results.
But surjectivity of $\rho_{m,k}$
is conveniently established here, by means of some modifications
of the right-inverses $K_1$, $K_n$. Details will be given for $k=1$;
to simplify notation, we treat $\rho_{m+1,1}$, so the 
trace of highest order is $\gamma_{m,1}$. 

The auxiliary function $\psi\in
\cfi C^\infty_0(\,]1,2[\,)$ with $\psi(0)=1$ can be taken such that also
$\psi'(0)=\dots=\psi^{(m)}(0)=0$. Indeed, we may arrange that
$\cf\psi(\xi_1)$ is orthogonal in $L_2(\,]1,2[\,)$ to 
$W_m:=\op{span}(\xi_1,\dots,\xi_1^m)$.
(It is well known that if a Hilbert space $H$ has a dense subspace $U$, it
holds for every subspace $W_m$ of dimension $m\in \N$ that 
$U\bigcap W_m^\perp$ is dense in the orthogonal complement $W_m^\perp$
(induction w.r.t.~$m$). In our case $f(\xi_1)\equiv1$ has projection 
$g\ne0$ onto $W_m^\perp$, so the density implies the existence of
$\phi\in C^\infty_0(\,]1,2[\,)\bigcap W_m^\perp$ such that 
$0\ne \int_1^2 \phi\overline{g}\,d\xi_1=\int_1^2\phi\overline{f}\,d\xi_1=
\int_1^2\phi\,d\xi_1=:c$.
Then we can take
$\psi=\tfrac{2\pi}{c}\cfi\phi$.)

Setting $\psi_k(x_1)=(k!)^{-1}x_1^k\psi(x_1)$ for 
$k\le m$, we have $\gamma_{j,1}\psi_k=
(\gamma_{j,1}x_1^k)\psi(0)/k!=\delta_{jk}$ (Kronecker delta). 
Using $\psi_\nu$, we let
\begin{equation}
 K_{\nu,1}v(x)=\sum_{j=0}^\infty 2^{-ja_1\nu}\psi_\nu(2^{ja_1}x_1)v_j(x'')
  \quad\text{for}\quad  \nu=0,1,\dots,m.
\end{equation} 
It holds that $K_{\nu,1}v$ is in $C(\R,\cs'(\Rnl))$ and $K_{\nu,1}$ is
continuous $\cs'(\Rnl)\to\cs'(\Rn)$,
for the arguments for $K_1$ apply verbatim, as $\psi_\nu$ amounts to a
special choice of $\psi$.
Moreover, since $\partial^\nu_{1}$ is $\cs'$-continuous, it applies
termwisely, which cancels the factor $2^{-ja_1\nu}$ and shows that
$\partial^\nu_{1}K_{\nu,1}v$ is in $C(\R,\cs'(\Rnl))$; i.e.\ 
$K_{\nu,1}$ maps into the domain of $\gamma_{\nu,1}$.
Incorporation of the factor $2^{-ja_1\nu}$ into the $K_1$-estimates 
yield continuity of 
$K_{\nu,1}\colon F^{s-\nu a_1-\frac{a_1}{p_1},a''}_{p'',p_1}\to
\fspq$ for all $s\in\R$, $0<q\le\infty$.

Finally $K^{(m+1)}_1=\left(\begin{matrix}
K_{0,1}& \dots& K_{m,1}\end{matrix}\right)$ maps $\cs'(\Rnl)^{m+1}$ into the
domain of $\rho_{m+1,1}$ and
fulfils $\rho_{m+1,1}\circ K^{(m+1)}_1=I$, 
since $\gamma_{k,1}K_{\nu,1}v=\delta_{k\nu} v$; 
and $K^{(m+1)}_1$ is continuous
with respect to the spaces in Corollary~\ref{gj1-cor}.

\subsection{Proof of Theorem~\ref{main1-thm}}
Note first that \eqref{cd'-cnd}$\implies$\eqref{spq-cnd} is the special case
$m=1$ of Lemma~\ref{spq-lem}, proved above.

For brevity we use the following notation for maximal functions invoking the
Littlewood--Paley decomposition,
\begin{equation}
  u^*_j(\vec t;x) =\sup_{y\in \Rn} 
   |\Phi_j(D)u(x-y)|\prod_{k=1,\dots,n} (1+|2^{ja_k}y_k|^{\frac1{t_k}})^{-1}.
\end{equation}
This applies via the estimate in Proposition~\ref{maximalf}, 
so it is once and for all assumed that $\vec t$ is chosen so that
$t_j<\min(p_1,\dots p_j,q)$ for all $j\ge1$.

\subsubsection{The basic mixed-norm estimates}
To see that \eqref{spq-cnd}$\implies$\eqref{cd'-cnd}, let 
$u\in F^{s,\vec a}_{\vec p,q}(\Rn)$ with 
$u=\sum_{j=0}^\infty u_j$ for $u_j=\Phi_j(D)u$, and let 
$\vec t$ be chosen as above. Then
\begin{equation}
  |u_j(0,x'')| \le c_1
                \frac{u_j(x_1-y_1,x'')}{1+|2^{ja_1}y_1|^{\frac1{t_1}}}
                \bigm|_{{y_1=x_1}}
 \le c_1 u^*_j(\vec t;(x_1,x'')),
  \label{uj*-ineq}
\end{equation}
since $1+|2^{ja_1}x_1|^{\frac1{t_1}}\le
1+2^{\frac{a_1}{t_1}}=:c_1$ for $x_1\in [2^{-ja_1},2^{(1-j)a_1}]$.
Next an integration yields
\begin{equation}
  (2^{a_1}-1)2^{-ja_1}|u_j(0,x'')|^{p_1}\le c_1^{p_1}
  \int_{2^{-ja_1}}^{2^{(1-j)a_1}} |u^*_j(\vec t;x)|^{p_1}\,dx_1,
  \label{uj*int-eq}
\end{equation}
so after multiplication by $2^{sjp_1}$ and estimation by $\sup_{k}2^{sk}
|u^*_k(\vec t;x)|$ in the integral, a summation yields
\begin{equation}
  \sum_{j=0}^\infty 2^{(s-\frac{a_1}{p_1})jp_1}
  |u_j(0,x'')|^{p_1}
  \le c'_1 
  \int_{\R} (\sup_k 2^{sk}|u^*_k(\vec t;x)|)^{p_1}\,dx_1.
\end{equation}
Then Proposition~\ref{maximalf} gives, since 
$\fspq\imb F^{s,\vec a}_{\vec p,\infty}$,
\begin{equation}
  \Norm{ (\sum_{j=0}^\infty |2^{(s-\frac{a_1}{p_1})j}
  u_j(0,x'')|^{p_1})^{\frac1{p_1}}}{L_{p''}}
  \le c_1''\norm{u}{F^{s,\vec a}_{\vec p,q}}.
  \label{basic-ineq}
\end{equation}
Moreover, by summing only over $j$
between $N+1$ and $N+m$ (and by applying the first part of
\eqref{maximalf-ineq} 
to a sequence of functions that vanish except for those $j$), one gets a
sharper conclusion, with 
$\chi_{N}$ as the characteristic function of $\,]0,2^{-Na_1}]$ 
and $v(x):=\sup_k 2^{sk}|u^*_k(x_1,x'')|$ 
for brevity, 
\begin{equation}
  \Norm{ (\sum_{j=N+1}^{N+m} |2^{(s-\frac{a_1}{p_1})j}
  u_j(0,\cdot)|^{p_1})^{\frac1{p_1}}}{L_{p''}}
  \le c_1''\Norm{ \chi_N(x_1)v(x)}{L_{\vec p}(\Rn)}
  \searrow 0.
  \label{basicN-ineq}
\end{equation}
The behaviour for $N\to\infty$ follows by
majorised convergence (with $v(\cdot,x'')$ as the first
majorant), since $c$ is independent of $N$.

For $s=\frac{a_1}{p_1}+\sum_{k>1}(\frac{a_k}{p_k}-a_k)_+$ 
we set 
$r_k=\max(1,p_k)$  so that
\begin{equation}
  s-\frac{a_1}{p_1}=\sum_{k>1>p_k} (\frac{a_k}{p_k}-a_k) 
  =\sum_{k>1} (\frac{a_k}{p_k}-\frac{a_k}{r_k}) =:\sigma.
  \label{sigma-eq}
\end{equation}
We continue in the same way for $\sigma>0$ and for $\sigma=0$.
The vector-valued
Nikol$'$skij inequality on $\Rnl$, cf.~Theorem~\ref{vNPP-thm},
then implies
\begin{equation}
  \begin{split}
  \Norm{\sum_{j=N+1}^{N+m} u_j(0,\cdot)}{L_{r''}}
  &\le
    \Norm{\sum_{j=N+1}^{N+m} |u_j(0,\cdot)|}{L_{r''}}
\\
  &\le
  c_{r''}\Norm{ (\sum_{j=N+1}^{N+m} |2^{j\sigma}
  u_j(0,\cdot)|^{p_1})^{\frac1{p_1}}}{L_{p''}}
  \le 
  c_{r''}c_1''\norm{\chi_N(x_1)v(x)}{L_{\vec p}}.
  \end{split}
\end{equation}
Consequently $\sum u_j(0,x'')$ converges in the Banach space
$L_{r''}(\Rnl)\imb \cs'(\Rnl)$ in all the borderline cases. 
(For $p_1\le1$ this can also be seen more directly, 
using that $\ell_{p_1}\imb\ell_1$ instead of the Nikol$'$skij inequality.)
By similar inequalities
now with summation over $j\in\N_0$, it is in both cases seen from
\eqref{basic-ineq}
that $\gamma_{0,1}$ is bounded $\fspq\to L_{r''}$.

The generic cases given by the sharp inequality 
$s>\frac{a_1}{p_1}+\sum_{k>1}(\frac{a_k}{p_k}-a_k)_+$
also give the desired $\cd'$-continuity, as seen by restricting 
$\gamma_{0,1}$ to subspaces with higher values of $s$.   

\subsubsection{Continuity in $x_1$}
To show that $\fspq(\Rn)\imb C_{\op{b}}(\R,L_{r''}(\Rnl))$ it is,
by a simple embedding lowering $s$, enough to treat
the case $s=\frac{a_1}{p_1}+\sigma$; cf~\eqref{sigma-eq}.
We may assume $q<\infty$, by passing to a larger space 
by means of a Sobolev embedding increasing 
a component of $p''$.

To evaluate at $x_1=z$ for an arbitrary $z$ one can extend the above estimates.
Indeed, letting $x_1$ run in $[z+2^{-ja_1},z+2^{(1-j)a_1}]$, 
and replacing $y_1$ by $y_1-z$, one finds \eqref{uj*int-eq} 
with an integral over this interval
(with the same constant).

This procedure gives the strengthened estimate
\begin{equation}
  \sup_z\Norm{\sum_{j=0}^{\infty} u_j(z,\cdot)}{L_{r''}}
  \le
  c_{r''}\sup_z\Norm{ (\sum_{j=0}^{\infty} |2^{j\sigma}
  u_j(z,\cdot)|^{p_1})^{\frac1{p_1}}}{L_{p''}}
 \le 
  c_{r''}c_1''\norm{u}{\fspq}.
\label{supz-eq}
\end{equation}
Redefining $u_j$ to $0$ for $j\notin [N+1,N+m]$, as before, 
this gives convergence of the series for every $z$, hence 
a function $z\mapsto f(z)=\sum u_j(z,\cdot)$, and \eqref{supz-eq} shows
it is bounded $\R\to L_{r''}$.

The continuity of $f$ follows because translations $\tau_h u\to u$ in 
$\fspq$ for $h\to 0$, since $q$ is finite; cf.\ Proposition~\ref{FBtrans-prop}.
Indeed, inserting $\tau_hu-u$ in \eqref{supz-eq},
\begin{equation}
  \norm{f(z-h)-f(z)}{L_{r''}}\le c\Norm{\tau_hu -u}{\fspq}\searrow 0.
\end{equation}

To show that $\Lambda_f=u$, note first that by \eqref{supz-eq} there is 
an estimate uniformly over a compact interval containing every $z$ 
appearing in $\supp\varphi$,
\begin{equation}
  |\dual{\sum_{j=0}^N u_j(z,\cdot)}{\varphi(z,\cdot)}_{\Rnl}|
  \le c\sup_z\norm{\varphi}{(L_{r''})^*}\norm{u}{\fspq}.
\end{equation}
With this as a majorisation, 
\begin{equation}
  \dual{\Lambda_f}{\overline{\varphi}}
=
  \int_\R \sum_{j=0}^\infty\dual{ u_j(z,\cdot)}{
        \overline{\varphi(z,\cdot)}}_{\Rnl}\,dz
  =\sum_{j=0}^\infty \iint u_j\overline{\varphi}\,dx''dz
  =\sum_{j=0}^\infty \dual{u}{\overline{\varphi_j}}
  =\dual{u}{\varphi}.
\end{equation}
Thence $u=\Lambda_f\in C_{\op{b}}(\R,L_{r''}(\Rnl))$ as desired.

\subsection{Boundedness in the $F$-scale (Theorem~\ref{main1-thm'})}
Departing from the proof of Theorem~\ref{main1-thm},
note that in the subspaces where $s>\frac{a_1}{p_1}+\sum_{k>1}
(\frac {a_k}{\min(1,p_2,\dots,p_k,q)}-a_k)$, 
the dyadic corona criterion applies, because $u_j(0,x'')$ by the
Paley--Wiener--Schwartz Theorem has its 
spectrum where $|\xi''|_{a''}\le 2^{j+1}$; 
cf.\ \cite[Rem.~3.4]{JJ00bsp}.
Therefore \eqref{basic-ineq} implies
\begin{equation}
  \Norm{\sum u_j(0,x'')}{F^{s-\frac{a_1}{p_1},a''}_{p'',p_1}}
  c\le \norm{u}{\fspq}.
\end{equation}
The surjectivity follows from the already proved Theorem~\ref{K-thm},
in view of the formula $\gamma_{0,1}\circ K_1v=v$, proved for all $v\in
\cs'(\Rnl)$, and the mapping properties of $K_1$. 

\subsection{Proof of Theorems~\ref{mainn-thm}, \ref{mainn-thm'}}
The implications of \eqref{cd'-cnd'} were accounted for directly after the
theorems by means of  Lemma~\ref{spq-lem}.

For the proof of \eqref{spq-cnd'}$\implies$\eqref{cd'-cnd} 
the argument from Theorem~\ref{main1-thm} applies, mutatis mutandis. 
Indeed, as in \eqref{uj*-ineq} one finds
$  |u_j(x',z)| \le c'_1 u^*_j(\vec t;(x',x_n))$ for a constant $c_1'$ 
independent of $z$; then one can take the
$L_{p'}(\Rnl)$-norm on both sides and
proceed with the argument for 
\eqref{uj*int-eq}--\eqref{basic-ineq}. Setting $r_k=\max(1,p_k)$ for $k<n$
and $\sigma=\sum_{k<n}(\frac{a_k}{p_k}-\frac{a_k}{r_k})$, this gives
for $s=\frac{a_n}{p_n}+\sigma$ and $p_n\le1$, when
the Nikol$'$skij inequality is applied for each $j\ge0$, 
\begin{equation}
  \sup_z(\sum \norm{u_j(\cdot,z)}{L_{r'}}^{p_n})^{\frac1{p_n}}
  \le
  c'\sup_z (\sum 2^{j\sigma p_n}
    \norm{u_j(\cdot,z)}{L_{p'}}^{p_n})^{\frac1{p_n}}
  \le c''\norm{u}{F^{s,\vec a}_{\vec p,q}}.    
  \label{basic-ineq'}
\end{equation}
Now 
$\norm{\cdot}{\ell_1}\le \norm{\cdot}{\ell_{p_n}}$ 
gives a finite norm series, hence 
convergence of $\sum_{j=0}^\infty u_j(\cdot,z)$ to some $f(z)$ in the 
Banach space $C_{\op{b}}(\R,L_{r'}(\Rnl))$. Clearly $\sup_z\norm{f(z)}{L_{r'}}
\le c''\norm{u}{\fspq}$.

As for $\gamma_{0,1}$
there is an identification $\Lambda_f=u$,
which yields $\fspq(\Rn)\imb C_{\op{b}}(\R,L_{r'}(\Rnl))$. 
In particular the working definition of 
$\gamma_{0,n}u$ is defined by evaluation at $z=0$.

In cases with $s=\varepsilon+\frac{a_n}{p_n}+\sigma$ for $\varepsilon>0$, 
the inequality \eqref{basic-ineq'} 
is modified by having on its left-hand side a norm in
$\ell^\varepsilon_{p_n}$. But since
$\norm{\cdot}{\ell_1}\le \norm{\cdot}{\ell^\varepsilon_{p_n}}$ whenever
$0<p_n<\infty$, the inclusion into $C_{\op{b}}(\R,L_{r'})$ is seen in the same way.
Altogether \eqref{spq-cnd'}$\implies$\eqref{cd'-cnd} holds in all cases.

\bigskip

When \eqref{spqn-cnd} holds, the dyadic ball
criterion for Besov spaces, cf.\ Lemma~\ref{B-lem}, applies
yielding continuity $F^{s,\vec a}_{\vec p,q}(\Rn)\to
B^{s-\frac{a_n}{p_n},a'}_{p',p_n}(\Rnl)$; here the surjectivity is a
consequence of the formula $\gamma_{0,n}\circ K_n=I$.
This completes the proof of Theorem~\ref{mainn-thm}.

\section{Final Remarks}
  \label{finl-sect}

To conclude, we note that also 
if we specialise to $\vec a=(1,\dots,1)$ and $\vec
p=(p,\dots,p)$, our results on the right-inverses
$K_j$ ($j=1$ and $j=n$) 
supplement those previously available, say in \cite[2.7.2]{T2},
since the $K_j$ are shown above to be well-defined continuous maps 
$\cs'(\R^{n-1})\to\cs'(\Rn)$.
Moreover, we show that all of $\cs'(\Rnl)$ is mapped into
the domain of $\gamma_{0,j}$, i.e. into $C(\R,\cd'(\Rnl))$.
This makes sense because we consider the distributional trace.

We also estimate the norms in $C_{\op{b}}(\R,L_{r''})$
etc.\ directly in terms of the $\fspq$-norms. 

Already Berkolaiko gave specific counterexamples
for the trace problem of mixed-norm spaces with $1<\vec p<\infty$. 
Our counterexamples show 
the necessity of raising the borderlines upwards when $0<p_k<1$ holds 
for one of the tangential variables $x_k$.

It should also be mentioned that we have a fairly complete theory, 
carrying over most of the well-known results for isotropic spaces to
the quasi-homogeneous mixed-norm spaces $\fspq$. In particular, 
for fixed $\vec p$, we cover all $s$ running in a maximal
open half-line. (However, 
traces of $\bspq$ were not described, although we do not envisage any 
difficulties in doing so with the methods of the present paper.)

The works on parabolic problems with traces of mixed-norm spaces
\cite{DeHiPr,Wei05} have for the lateral boundary data used spaces
that are intersections of $F^{2-1/q}_{p,q}(\,]0,T[\,;
L_q(\partial\Omega))$
and $L_p(\,]0,T[\,; W^2_q(\Omega))$; also vector-valued solutions 
have been treated. 
We have left both questions (identifications of $\fspq$ spaces with 
intersections and vector-valued versions) for the future.

\subsection*{Acknowledgement}
The authors are grateful to Mathematisches Forschungsinstitut Oberwolfach
for a two weeks stay in the spring of 2006 
under the programme Research in pairs.

%
%

\def\cprime{$'$}
\providecommand{\bysame}{\leavevmode\hbox to3em{\hrulefill}\thinspace}

\end{document}